\documentclass[twoside]{amsart}
\usepackage{amsmath,amssymb,amscd,amsthm,amsxtra,amsfonts}
\usepackage{mathtools,mathrsfs,dsfont,xparse}
\usepackage{graphicx,epstopdf}
\usepackage{wrapfig}

\usepackage{multirow}
\usepackage{cases}
\usepackage{enumerate}
\usepackage{xcolor}
\usepackage{tikz,pgfplots}
\usetikzlibrary{matrix}
\usepackage{mathtools}
\usepackage[margin=1 in]{geometry}
\usepackage{todonotes}
\usepackage{listings}
\usepackage{tcolorbox}
\usepackage{caption}
\mathtoolsset{showonlyrefs}

\DeclareFontFamily{U}{mathx}{\hyphenchar\font45}
\DeclareFontShape{U}{mathx}{m}{n}{
  <5> <6> <7> <8> <9> <10> <10.95> <12> <14.4> <17.28> <20.74> <24.88> mathx10
}{}
\DeclareSymbolFont{mathx}{U}{mathx}{m}{n}
\DeclareMathAccent{\widecheck}{0}{mathx}{"71}

\allowdisplaybreaks[4]

\makeatletter
\newcommand{\dotDelta}{{\vphantom{\Delta}\mathpalette\d@tD@lta\relax}}
\newcommand{\d@tD@lta}[2]{%
  \ooalign{\hidewidth$\m@th#1\mkern-1mu\cdot$\hidewidth\cr$\m@th#1\Delta$\cr}%
}
\makeatother

\date{}

\usepackage{multicol,bbm}
\usepackage[colorlinks=true, pdfstartview=FitV, linkcolor=blue, citecolor=blue, urlcolor=blue]{hyperref}

\numberwithin{equation}{section}

\usepackage{cjhebrew}

\DeclareMathOperator{\im}{Im}

\usepackage[numbers,sort]{natbib}

{\theoremstyle {definition} \newtheorem {definition} {Definition} [section] }
{\theoremstyle {plain}  \newtheorem {theorem} [definition] {Theorem}}
{\theoremstyle {plain}  }
{\theoremstyle {plain} \newtheorem {proposition} [definition]{Proposition}}
{\theoremstyle {plain} \newtheorem {lemma}[definition] {Lemma}}
{\theoremstyle {definition} \newtheorem {remark}[definition] {Remark}}
{\theoremstyle {plain} }

\def\R{{\mathbb{R}}}

\def\e{{\varepsilon}}
\def\N{{\mathbb{N}}}
\def\Z{{\mathbb{Z}}}

\newcommand{\abs}[1]{\left|#1\right|}
\newcommand{\norm}[1]{\left\|#1\right\|}

\begin{document}

\author[Lee and Yu]{Zachary Lee and Xueying Yu}

\address{Zachary Lee 
\newline 
\indent Department of Mathematics, The University of Texas at Austin\indent 
\newline \indent 2515 Speedway, PMA 8.100
Austin, TX 78712 \indent
}
\email{zl9868@utexas.edu}

\address{Xueying  Yu
\newline \indent Department of Mathematics, Oregon State University\indent 
\newline \indent  Kidder Hall 368
Corvallis, OR 97331 \indent 
}
\email{xueying.yu@oregonstate.edu}

\title[GWP for septic NLS on $\mathbb{R}$]{Low-regularity global well-posedness for the one-dimensional, defocusing, septic nonlinear Schr\"odinger equation}



\begin{abstract}
In this paper, we show that the one dimensional septic nonlinear Schr\"odinger equation is globally well-posed in $H^s (\mathbb{R})$ when $s > 39/89$ with accompanying polynomial bounds in time. We utilize bilinear estimates to obtain improved bounds on the modified energy increment and prove new $L^{p+3}_{t,x}$ almost Morawetz estimates for the defocusing $p$-NLS adapted to the low-regularity setting, before specializing to the septic case $p=7$. These tools together allow us to reach Sobolev regularities $s<1/2$, showing for the first time a global existence result in the mass supercritical regime for a class of unbounded, discontinuous data. 

\end{abstract}

\maketitle

\setcounter{tocdepth}{1}
\tableofcontents

\parindent = 10pt     
\parskip = 8pt

\section{Introduction}

In this paper, we study the initial value problem for the septic defocusing nonlinear Schr\"odinger equation (NLS) in one dimension:
\begin{align}\label{septic NLS}
\begin{cases}
i\partial_t u + \Delta u = \abs{u}^{6} u \\
u(0,x)=u_0(x)\in H^s(\R).
\end{cases}
\end{align}
This problem is known to be locally well-posed in $H^s(\R), s\ge s_c:=1/6$ \cite{CazenaveBook}. The critical Sobolev regularity $s_c$ is determined by the invariance of the $\dot{H}^{s_c}(\R)$ norm under the rescaling
\begin{align}
    u(t,x)\mapsto  \lambda^{-1/3} u(t/\lambda^2, x/\lambda)
\end{align}
which leaves the class of solutions to $\eqref{septic NLS}$ invariant for any $\lambda>0$. The equation $\eqref{septic NLS}$ enjoys the following conserved quantities, if finite for the initial data,
\begin{align}
    \text{Mass} & := M(u(t))= \int_{\R} |u(t,x)|^2 \,dx = M(u(0)) ,\\
    \text{Energy} & := E(u(t))= \frac{1}{2}\int_{\R} |\nabla u (t,x)|^2 \,dx + \frac{1}{8}\int_{\R}  |u|^8\,dx= E(u(0)) ,\\
    \text{Momentum} & := P(u(t))= \im \int_\R \overline{u(t,x)}\partial_x u(t,x) \,dx=P(u(0)).
\end{align}
Note that the first two are coercive for the $L^2_x$ and the $\dot{H}^1_x$ norm respectively, but the last one is not coercive for the $\dot{H}^{1/2}_x$ norm. The local theory iterates to produce global well-posedness for initial data $u_0(x)\in H^1_x(\R)$ \cite{Cazenave1988TheCP}. In this case, it is known that these global solutions are bounded in the scaling-invariant Strichartz norm $L^9_{t,x}$ and scatter. It is conjectured that global well-posedness and scattering hold for initial data as rough as $u_0\in \dot{H}^{1/6}_x(\R)$.\par
There has been much interest in extending local well-posedness results for $s_c\le s <1$, where energy conservation cannot be used directly, to global well-posedness. Bourgain first introduced the Fourier truncation approach \cite{Bou98} to resolve such questions; see also 
\cite{GC06, SY20, KPV00, DET16}. Later, the I-method would be developed by Colliander, Keel, Staffilani, Takaoka, and Tao as a refinement of this technique to prove global well-posedness results for low-regularity initial data for the NLS in two and three dimensions \cite{CKSTT02, CKSTT04, CKSTT08}; see also 
\cite{CR11, DPSN08, DPSN07, DPSN072, FG07, Han12, LWX11, Su12, SY22}. This method was adapted to show global existence of rough solutions to the quintic, defocusing one-dimensional NLS in \cite{Tzi05}, followed by \cite{DPSN08} and later in \cite{Dod12}.

\par 
In Colliander--Holmer--Visan--Zhang \cite{CHVZ08}, it was proved that 
\begin{proposition}[\cite{CHVZ08}]
The Cauchy initial-value problem \eqref{septic NLS}
is globally well-posed and scatters for initial data $u_0 \in H^s(\R)$, provided $s > 8/13 \approx 0.\overline{615384}.$ In particular, there exist $u_{\pm} \in H^s(\R)$ such that
\begin{align}
\norm{u(t) -e^{it\Delta} u_{\pm} }_{H^s(\R)} \rightarrow 0 \quad \text{ as } \quad t\rightarrow \pm \infty .
\end{align}
\end{proposition}
Note by the Sobolev embedding, the above result only applies to data that is both continuous and bounded. As far as the authors know, there is no known global well-posedness result available for mass-supercritical NLS on $\R$ with discontinuous or unbounded data. In this paper, we present the first known result of this kind and make further progress towards the aforementioned conjecture with 
\begin{theorem}[Main Theorem]\label{Main Theorem}
The Cauchy initial-value problem \eqref{septic NLS} is globally well-posed for initial data $u_0 \in H^s(\R), 1>s>39/89 \approx 0.438$ with the following polynomial bounds:
\begin{align} \label{Final H^s bounds}
 \sup_{t\in[0,T]}\| u(t)\|_{H^s_x} \le C(s, \|u_0\|_{H^s_x}, \delta) (1+T)^{\frac{6s(1-s)}{89s-39-\delta}}
\end{align}
where $0<\delta<89s-39$ is arbitrary.

\end{theorem}
\begin{remark}
    Since the Sobolev exponent $39/89$ is less than $1/2$, the above theorem shows global well-posedness for a class of data which includes discontinuous and unbounded functions.
\end{remark}
\begin{remark}
Our main interest is to lower the global existence threshold exponent to $s<1/2$, and as a consequence we only obtain a polynomial bound in time and no scattering result. If we were to use an $L^{9}_{t,x}$ adapted Morawetz estimate obtained by interpolating the $L^8_{t,x}$ Morawetz estimate used in \cite{CHVZ08} and the $L^{10}_{t,x}$ Morawetz estimate proved in \cite{PV09}, then we can also improve the scattering threshold proved in \cite{CHVZ08} to $s>6/11\approx 0.545$.
\end{remark}
\begin{remark}
    We expect our method to also work for NLS in one dimension with higher nonlinearities $u|u|^{2k}$ for $k\ge 4$. In fact, scaling considerations show that the genuine $L^{2k+4}_{t,x}$ Morawetz estimate may be used to close the argument and therefore scattering is expected to follow for $s>s(k)$ with $s(k)$ that is lower than the known threshold in \cite{CHVZ08}. However, a quick calculation shows that for all such $k$, we have that $s(k)>1/2$, and therefore this scattering result is restricted to bounded, continuous data by the Sobolev embedding. We therefore choose to focus on the case $k=3$ since although we lose the scattering, we can show global existence of the solution for a class of discontinuous, unbounded data.
\end{remark}
\begin{remark}
    Note that by the usual approximation argument, it is enough to prove \eqref{Final H^s bounds} for $u_0\in \mathcal{S}(\R)$. For such data, which are in particular in the energy class, it is known that \eqref{septic NLS} is globally well-posed and the solution $u(t)$ will remain in $\mathcal{S}(\R)$ for all positive times.
\end{remark}

We prove this result by taking inspiration from the ideas contained in \cite{CHVZ08}, \cite{Dod12}. The new ingredient in our method is a space-time $L^{10}_{t,x}$ almost Morawetz estimate whose derivation is inspired by the $L^{10}_{t,x}$ Morawetz estimate proved in  \cite{PV09}. We face the following difficulties in proving Theorem \ref{Main Theorem}:
\begin{enumerate}
\item Compared with \cite{CHVZ08}, we are in a situation where the scaling of the $L^{10}_{t,x}$ Morawetz norm is incompatible with a local well-posedness result like that of Lemma \ref{lem LWP}, which is crucial to the rest of the argument. This is unlike the $L^8_{t,x}$ Morawetz norm whose scaling allows the authors of \cite{CHVZ08} to prove such a lemma. To remedy this, we must instead work with the $L^8_t L^{10}_x$ norm, and then use H\"older in time to relate the two norms in the bootstrap argument in Section \ref{Proof of Main Theorem}. We were motivated by the similar method that Dodson took in \cite{Dod12} to resolve a problem of the same type. Consequently, we use a nested bootstrap argument to prove the main theorem in the final section of the paper, where inside the bootstrap argument for controlling $E(Iu(t))$, we perform yet another bootstrap argument to control the Morawetz norm of $Iu$. Unfortunately, this leads to a polynomial upper bound on the $H^s(\R)$ norm of the solution, which precludes a proof of scattering. 
\item Compared to \cite{Dod12}, which deals with the mass-critical case, in estimating the energy increment, we can only use in general the homogeneous Strichartz bound
\begin{align}
    \| \nabla Iu\|_{L^p_t L^q_x} \lesssim \|Iu\|_{\dot{H}^1_x}
\end{align}
instead of its inhomogeneous cousin
\begin{align}
    \| \langle \nabla\rangle Iu\|_{L^p_t L^q_x} \lesssim \|Iu\|_{H^1_x}
\end{align}
which is what Dodson in \cite{Dod12} used in his control of the energy increment. This is due to the fact that working with the inhomogeneous version would mean losing large powers of the mass, which under rescaling is supercritical, and this would doom the method. Instead, we only lose powers of the mass when they are accompanied by a negative power of $N$, as may happen for example when a bilinear estimate is used. This means that dealing with the low frequencies of $u$ is particularly delicate and to estimate these, we rely on control of the adapted $L^8_t L^{10}_x$ Morawetz norm and the ``defocusing'' fixed-time bound $\|Iu(t)\|_{L^8_x}\lesssim  E(Iu(t))^{1/8}$ in combination with Bernstein inequalities.
\end{enumerate}

Since \eqref{septic NLS} is subcritical in $H^s$ for $s>1/6$, the length of the local well-posedness interval depends only on the norm $\norm{u_0}_{H^s_x(\R)}$ of the solution at the initial time. Hence, global well-posedness follows if we can establish a uniform bound, as we do in this paper, on the $H^s$ norm of $u$ on an arbitrary time interval. However, the $H^s$ norm is not conserved, so we must smoothen out the initial data by way of the $I$-operator, whose $C^\infty$ multiplier is given by
\begin{align}\label{eq m}
    m_N(\xi)=\begin{cases}
        1 & |\xi|\le N ,\\
        \left(N/|\xi|\right)^{1-s} & |\xi|\ge 2N,
    \end{cases}
\end{align}
with
\begin{align}\label{eq I}
I_N(D) f \equiv m_N(D) f,
\end{align}
where $D=-i\partial_x$,
and control the $H^s$ norm of $u$ by $E(I(u))$ the energy of $Iu$, which we show is \textit{almost} conserved. The increment in the energy of $Iu$ is shown to be controlled on intervals where the $L^8_t L^{10}_{x}$ norm of $Iu$ is small. Conversely, the Morawetz norm of $Iu$ is shown to be bounded above by the energy of $Iu$. This sets us up for a bootstrap argument in Section \Ref{Proof of Main Theorem}.
\par 
To discuss the Morawetz estimates, we say that $u(t,x)$ solves the one-dimensional defocusing $p$-NLS with initial data $u_0$ if
\begin{align}\label{p-NLS}
    \begin{cases} 
    i\partial_t u + \Delta u = |u|^{p-1} u\\
    u(0,x)=u_0.
    \end{cases}
\end{align}
The following Morawetz estimate was proved in \cite{CHVZ08}:
\begin{lemma}
    Let $u(t,x)$ solve \eqref{p-NLS} with    $p>1$. Then, we have that
    \begin{align}
            \norm{u}_{L^{8}_{t,x}([0,T]\times \R)}^{8} \lesssim \norm{u_0}_{L^2 (\R)}^6   \norm{u(t)}_{L^\infty_t\dot{H}^{1/2}_x ([0,T]\times \R)}^2.
    \end{align}
\end{lemma}
We will in this paper make use of the following Morawetz estimate for solutions to the defocusing $p$-NLS proved in \cite{PV09} that is useful for $u$ in $\dot{H}^{1/2}_x(\R)$ on the slab $[0,T]\times \R$.
\begin{lemma}[Morawetz Estimate in \cite{PV09}] If $u$ solves \eqref{p-NLS} with $p> 1$, then 
\begin{align}
    \norm{u}_{L^{p+3}_{t,x}([0,T]\times \R)}^{p+3} \lesssim_p \norm{u_0}_{L^2 (\R)}^2   \norm{u(t)}_{L^\infty_t\dot{H}^{1/2}_x ([0,T]\times \R)}^2.
\end{align}
\end{lemma}
In order to prove the main theorem, we show that a `low-regularity' Morawetz estimate holds for solutions  $u$ to the one-dimensional defocusing $p$-NLS (see Proposition \ref{lem Morawetz}). We will later specialize to $p=7$ to treat the septic case. Note the right hand side of the $L^{p+3}_{t,x}$ Morawetz estimate contains fewer copies of the mass, which in our mass supercritical regime leads to a much better control on the Morawetz norm. In turn, this gives improved bounds for the energy increment on long time intervals, leading to global existence of solutions for rougher data. 
\par
In Section \Ref{Preliminaries}, we outline some notation and reference Strichartz estimates and Littlewood--Paley theory. In Section \Ref{Bilinear Estimates}, we present the classical bilinear estimate and an interpolated version that we will use several times throughout the paper. In Section \Ref{Local Well-Posedness}, we prove a version of local well-posedness involving the Morawetz and Strichartz norms of $Iu$. In Section \Ref{Energy Increment}, we prove the upper bound for the increment in $E(Iu(t))$. In Section \Ref{Almost Morawetz Estimate}, we prove our new almost Morawetz estimate. Finally, in Section \Ref{Proof of Main Theorem}, we prove Theorem \Ref{Main Theorem}.

\subsection*{Acknowledgment}
Z. Lee was partially supported by the Jack Kent Cooke Foundation, the NSF grants DMS-1840314, DMS-2052789 and DMS-2511517, and The University of Texas at Austin through the Provost’s Graduate Excellence Fellowship. X. Yu was partially supported by NSF DMS-2306429. Part of this work was completed during X. Yu's participation in the SLMath program ``Recent Trends in Stochastic Partial Differential Equations".

\section{Preliminaries}\label{Preliminaries}
We use the notation $X\lesssim Y$ to mean that $X\le CY$ for some $C>0$. If both $X\lesssim Y$ and $Y\lesssim X$, then we write $X\sim Y$. If $C$ depends on a parameter $\alpha$, we will write $X\lesssim_\alpha Y$. We write 
\begin{align}
    X\lesssim Y^{a+}
\end{align}
if for any sufficiently small $\delta>0$,
\begin{align}
    X\lesssim_\delta Y^{a+\delta}.
\end{align}
\par 
Given an interval $J$, we define the mixed space-time norms
\begin{align}
    \norm{u}_{L^p_t L^q_x(J\times \R)} = \left(\int_J \left(\int_\R |u(t,x)|^q\,dx\right)^{p/q}\,dt\right)^{1/p}
\end{align}
with the usual modification if $p$ or $q$ is infinite. When $p=q$, we will often denote the norm as $\norm{\cdot}_{L^p_{t,x}(J\times \R)}$.
Note that we sometimes abbreviate these norms as $\norm{u}_{p,q}$ when the interval is understood.  \par 
With $\widehat{f}$ for the Fourier transform of $f$, we define the following fractional differentiation operators,
\begin{align}
    \widehat{|\nabla|^s f}&= |\xi|^s \hat{f},\\
    \widehat{\langle \nabla \rangle^s f} &= \langle \xi\rangle^s \hat{f},
\end{align}
with the usual Sobolev norms
\begin{align}
    \norm{f}_{\dot{H}^s(\R)}&= \norm{|\nabla|^s f}_{L^2(\R)},\\
    \norm{f}_{H^s(\R)}&= \norm{\langle \nabla\rangle ^s f}_{L^2(\R)}.
\end{align}
\begin{definition}
    A pair of exponents $(p,q)$ is called \textit{Strichartz-admissible} if
    \begin{align}
        \frac{2}{p}+\frac{1}{q}=\frac{1}{2}, \quad 2\le q\le \infty.
    \end{align}
    For an interval $J$, we define the Strichartz norm
    \begin{align}
        \norm{f}_{S^0(J\times \R)} = \sup_{(p,q) \text{ admissible } } \norm{f}_{L^p_t L^q_x (J\times \R)}.
    \end{align}
\end{definition}
We have the celebrated Strichartz estimates.
\begin{lemma}[See \cite{Keel-Tao, GV92, Strichartz77}] \label{Strichartz estimates Lemma}
    If $u$ solves on an interval $J=[t_0, t_1]$
    \begin{align}
    \begin{cases}
     i\partial_t u+\Delta u = F(t,x) ,\\
        u(t_0,x)=u_{t_0},   
    \end{cases}    
    \end{align}
    then
    \begin{align}
        \norm{u}_{S^0(J\times \R)} \lesssim \norm{u_{t_0}}_{L^2(\R)} + \norm{F}_{L^{p'}_t L^{q'}_x(J\times \R)}
    \end{align}
    for any $(p',q')$ dual to an admissible pair $(p,q)$.
\end{lemma}
When the interval $J$ is understood, we will sometimes write $S^0$ for $S^0(J\times \R)$. \par 
Let $\varphi(\xi)$ be a smooth even bump function supported in $|\xi|\le 2$ and equalling 1 on $|\xi|\le 1$. For each $N\in 2^{\Z}$, we define the Littlewood--Paley operators
\begin{align}
    \widehat{P_{\le N} f}(\xi)  &= \varphi(\xi/N) \hat{f}(\xi) \\
    \widehat{P_{> N} f}(\xi)  &= (1-\varphi(\xi/N)) \hat{f}(\xi) \\
    P_{ N} f&= (P_{\le N}-   P_{\le N/2}) f.
\end{align}
We present the following Bernstein/Sobolev estimates that we take from \cite{CHVZ08}.

\begin{lemma}
For any $1 \leq p \leq q \leq \infty$ and $s>0$, we have
\begin{align}
\left\|P_{\geq N} f\right\|_{L_x^p} & \lesssim N^{-s}\left\||\nabla|^s P_{\geq N} f\right\|_{L_x^p} \\
\left\||\nabla|^s P_{\leq N} f\right\|_{L_x^p} & \lesssim N^s\left\|P_{\leq N} f\right\|_{L_x^p} \\
\left\||\nabla|^{ \pm s} P_N f\right\|_{L_x^p} & \sim N^{ \pm s}\left\|P_N f\right\|_{L_x^p} \\
\left\|P_{\leq N} f\right\|_{L_x^q} & \lesssim N^{\frac{1}{p}-\frac{1}{q}}\left\|P_{\leq N} f\right\|_{L_x^p} \\
\left\|P_N f\right\|_{L_x^q} & \lesssim N^{\frac{1}{p}-\frac{1}{q}}\left\|P_N f\right\|_{L_x^p} .
\end{align}
\end{lemma}
We will use the following Littlewood--Paley square function estimate
\begin{lemma}\label{LP estimate}
The following estimate holds
\begin{align}
    \norm{\left(\sum_{N\in 2^\Z} N^{2s}| P_N u|^2\right)^{1/2}}_{L^p_x(\R)} \sim_{p,s} \norm{|\nabla|^s u}_{L^p_x(\R)}
\end{align}
for any $1<p<\infty$ and $s\ge 0$.
\end{lemma}
We collect important properties of the $I$-operator in the following lemma.
\begin{lemma}[See \cite{CHVZ08}]\label{I-operator lemma}
    Let $1<p<\infty$ and $0\le \sigma\le s \le 1$. Then, 
    \begin{align}
        \norm{If}_p &\lesssim \norm{f}_p, \\
        \norm{|\nabla|^\sigma P_{>N} f}_p &\lesssim N^{\sigma-1} \norm{\nabla I f}_{p}, \\
        \norm{f}_{H^s} \lesssim \norm{If}_{H^1} & \lesssim N^{1-s} \norm{f}_{H^s}. 
    \end{align}
\end{lemma}

\subsection{Bilinear Estimates} \label{Bilinear Estimates}
In this subsection, we collect some bilinear estimates that will be used in the rest of the paper. 

We first state the classical bilinear estimate:
\begin{lemma}[See \cite{KT07, Tao2006}]\label{lem bilinear}
Let $J$ be an interval, and $u_1, u_2$ be two solutions to, respectively, 
\begin{align}
i\partial_t u_1 + \Delta u_1 = F_1, \qquad i\partial_t u_2 + \Delta u_2 = F_2,
\end{align}
such that $u_1$ is supported on $|\xi|\sim M$ and $u_2$ on $|\xi|\sim N$ with $M\lesssim N$ and the two separated in frequency space by at least $cN$. Then, 
    \begin{align}
        \norm{u_1 u_2}_{L^{2}_{t,x}(J\times \R)} \lesssim \frac{1}{N^{1/2}} \norm{u_1}_{S^0_*(J\times \R)}  \norm{u_2}_{S^0_*(J\times \R)},
    \end{align}
    where
    \begin{align}
        \norm{u_j}_{S^0_*(J\times \R)} = \norm{u_j(t_0)}_{L^2(\R)} +  \norm{F_j}_{N(J\times \R)},
    \end{align}
where $t_0\in J$ is the initial time, and
\begin{align}
    \norm{F}_{N(J\times \R)} \equiv \inf_{(p,q)\operatorname{ 
 admissible} } \norm{F}_{L^{p'}_t L^{q'}_x (J\times \R)}.
\end{align}
\end{lemma}
We also need the following modified bilinear estimate proved in \cite{Dod12}.
\begin{lemma}\label{lem bilinear Dodson}
Let $u,v$ solve $i\partial_t u+\Delta u=F_1$ and $i\partial_t v+\Delta v=F_2$ on $J\times \R$. Then, for $N\in 2^{\Z}$,
\begin{align}
    \norm{P_N (u\overline{v})}_{L^2_{t,x}(J\times \R)} \lesssim \frac{1}{N^{1/2}} \norm{u}_{S^0_*(J\times \R)}  \norm{v}_{S^0_*(J\times \R)}.
\end{align}
\end{lemma}
\section{Local Well-Posedness}\label{Local Well-Posedness}
In this section, we prove a local well-posedness result we will frequently use in the following sections showing a certain $Z(J)$ norm (see below) is controlled on intervals $J$ where the adapted Morawetz $L^8_t L^{10}_x(J\times \R)$ norm of $Iu$ is small. \par 
Following \cite{CHVZ08}, we first define
\begin{align}
    \| u \|_{Z(J)} := \|  P_{\le 1} u\|_{S^0(J\times \R)} + \sup_{(p,q)\neq (4,\infty) \text{ admissible }} \left( \sum_{N\in 2^{\Z}, N>1} \| P_N u\|_{L^p_t L^q_x (J\times  \R)}^2 \right)^{1/2}
\end{align}
We note the following inequality, 
\begin{align}
    \left\| \left(\sum_{N} |P_N u|^2\right)^{1/2} \right\|_{L^p_t L^q_x(J\times \R)} \le  \left(\sum_{N} \| P_N u \|_{L^p_t L^q_x(J\times \R)}^2 \right)^{1/2} \quad \, \forall \, 2\le p,q\le \infty, 
\end{align}
which may be proved via Minkowski's inequality. The above inequality together with the Littlewood--Paley square function estimate yields (here we use the fact that $2\le q<\infty$), 
\begin{align} \label{Z norm inequality}
    \|u\|_{L^p_t L^q_x(J\times \R)} \lesssim \left\| \left(\sum_{N} |P_N u|^2\right)^{1/2} \right\|_{L^p_t L^q_x(J\times \R)} \le \left(\sum_{N} \| P_N u \|_{L^p_t L^q_x(J\times \R)}^2 \right)^{1/2}. 
\end{align}
Minkowski's inequality and the Littlewood--Paley inequality yield that for a dual-admissible pair $(p',q')\neq (4/3,1)$, we have that
\begin{align} \label{dual of Z norm inequality}
    \left(\sum_{N} \| P_N g\|_{L^{p'}_t L^{q'}_x(J\times \R)}^2\right)^{1/2} \lesssim \| g\|_{L^{p'}_t L^{q'}_x(J\times \R)}
\end{align}
We now present the main result of this section.
\begin{lemma}\label{lem LWP}
    Let $u(t)$ be a strong solution to \eqref{septic NLS} on the spacetime slab $J\times \R$, where  $J=[t_1, t_2]$ is an interval with $E(Iu(t_1))\le 1$. Additionally, assume that
    \begin{align}\norm{Iu}_{L^{8}_{t}L^{10}_x (J\times \R)} \le \e 
    \end{align}
   for sufficiently small $\varepsilon$. Then, for $N$ large enough and $s>1/6$,
    \begin{align}
        \norm{\nabla  I u}_{Y(J)} \lesssim \norm{\nabla Iu(t_1)}_{L^2}, 
    \end{align}
    where
    \begin{align}
        \norm{\nabla  I u}_{Y(J)} := \norm{\nabla  I u}_{Z(J)} + \norm{\nabla  I u}_{S^0(J)}. 
    \end{align}
    
\end{lemma}
\begin{remark}
    Note that it is important that we control both the $Z(J)$ and the $S^0(J)$ norms since in certain parts of the proof of Theorem \ref{Energy Increment Theorem}, both norms are needed to control the resulting terms. 
\end{remark}
\begin{proof}[Proof of Lemma \ref{lem LWP}]
    By translation invariance, we may assume that $t_1=0$. 
    We have that by the definition of the $Z(J)$ norm and the Duhamel formula for strong solutions, 
    \begin{align} \label{Z norm of Duhamel}
        \|\nabla Iu\|_{Z(J)} \le \| \nabla I e^{it\Delta} u_0 \|_{Z(J)} + \left\| \int_0^t e^{i(t-s)\Delta} \nabla I(|u|^6 u)(s)\,ds \right\|_{Z(J)}.
    \end{align}
    We first estimate the first term on the right hand side of \eqref{Z norm of Duhamel} using Lemma \ref{Strichartz estimates Lemma}, Lemma \ref{LP estimate} with $p=2, s=0$, and \eqref{Z norm inequality},
    \begin{align}
        \| \nabla I e^{it\Delta} u_0 \|_{Z(J)} &= \| e^{it\Delta} \left( \nabla I  P_{\le 1} u_0\right) \|_{S^0(J\times \R)} + \sup_{(p,q)\neq(4,\infty)\operatorname{admissible}} \| \| \nabla I e^{it\Delta} P_N u_0 \|_{L^p_t L^q_x(J\times \R)} \|_{\ell^2_N(2^{\N+1})} \\
        &\le \| \nabla I  P_{\le 1} u_0\|_{L^2_x} + C \left\| \| \nabla I P_N u_0 \|_{L^2_x} \right\|_{\ell^2_N(2^{\N+1}) } \\
        &\le \| \nabla I  P_{\le 1} u_0\|_{L^2_x} + C \left\| \left(\sum_{N\in 2^{\N+1}} |P_N \nabla I u_0|^2\right)^{1/2} \right\|_{L^2_x} \\
        &\lesssim \| \nabla I u_0\|_{L^2_x}. 
    \end{align}
    Similarly,
    \begin{align}
        \| \nabla I e^{it\Delta} u_0 \|_{S^0} \lesssim  \| \nabla I u_0\|_{L^2_x}.
    \end{align}
    We now estimate the second term on the right hand side of \eqref{Z norm of Duhamel} using Lemma \ref{Strichartz estimates Lemma}, \eqref{dual of Z norm inequality} and the Littlewood--Paley inequality with exponent $q'>1$, where $(p',q')$ is any pair dual to an admissible pair $(p,q)$, 
    \begin{align}
        &\quad \left\| \int_0^t e^{i(t-s)\Delta} \nabla I(|u|^6 u)(s)\,ds \right\|_{Z(J)} \\
        &= \left\|  \int_0^t e^{i(t-s)\Delta} \nabla I P_{\le 1}(|u|^6 u)\,ds \right\|_{S^0(J\times \R)} + \sup_{(p,q)\neq(4,\infty)\operatorname{admissible}} \left\| \left\| \int_0^t e^{i(t-s)\Delta} \nabla IP_N(|u|^6 u)(s)\,ds \right\|_{L^p_t L^q_x(J\times \R)} \right\|_{\ell^2_N(2^{\N+1})} \\
        &\le C\| \nabla I P_{\le 1}(|u|^6 u)\|_{L^{p'}_t L^{q'}_x} + C\left\| \|\nabla I P_{N}(|u|^6 u) \|_{L^{p'}_t L^{q'}_x} \right\|_{\ell^2_N(2^{\N+1})} \\
        &\le C\| \nabla I P_{\le 1}(|u|^6 u)\|_{L^{p'}_t L^{q'}_x} + C\left\| \|\nabla I P_{N}(|u|^6 u) \|_{\ell^2_N(2^{\N+1}) } \right\|_{L^{p'}_t L^{q'}_x} \\
        &\le C\| \nabla I P_{\le 1}(|u|^6 u)\|_{L^{p'}_t L^{q'}_x} + C\left\| \nabla I P_{>1} (|u|^6 u) \right\|_{L^{p'}_t L^{q'}_x} \\
        &\lesssim \| \nabla I (|u|^6 u)\|_{L^{p'}_t L^{q'}_x}
    \end{align}
    where we used the Littlewood--Paley inequality. 
    We also have that
    \begin{align}
        \left\| \int_0^t e^{i(t-s)\Delta} \nabla I(|u|^6 u)(s)\,ds \right\|_{S^0(J)} \le \| \nabla I (|u|^6 u)\|_{L^{p'}_t L^{q'}_x}.
    \end{align}
    We now let $p'=1, q'=2$. By the fractional Leibniz rule, 
    \begin{align*}
        \norm{\nabla I (|u|^6 u)}_{L^{1}_t L^{2}_x} 
        &\lesssim \norm{\nabla Iu}_{L^{4}_t L^{\infty}_x} \norm{|u|^6}_{L^{4/3}_t L^2_x} \\
        &\le  \norm{\nabla Iu}_{S^0(J)} \norm{u}^6_{L^8_t L^{12}_x} \\
        &\le \norm{\nabla Iu}_{Y(J)} \norm{u}^6_{L^8_t L^{12}_x}.
    \end{align*}
We now consider separately $\norm{P_{\le 1} u}_{L^8_tL^{12}_x}$, $\norm{P_{1< \cdot \le N } u}_{L^8_tL^{12}_x}$ and $\norm{P_{>N } u}_{L^8_tL^{12}_x}$. For the low frequencies, we use  Bernstein to obtain
\begin{align}
\norm{ u_{\le 1} }_{L^8_t L^{12}_x} &\lesssim   \norm{ Iu_{\le 1} }_{L^8_t L^{10}_x} \lesssim \varepsilon 
\end{align}
For the medium frequencies, we use interpolation  to show that, with $\theta=5/6$, 
\begin{align}
    \norm{ u_{1<\cdot\le N} }_{L^8_t L^{12}_x} &\le  \norm{ Iu_{1<\cdot\le N} }_{L^8_t L^{10}_x}^{\theta} \norm{ Iu_{1<\cdot\le N} }_{8 ,\infty}^{1-\theta} \lesssim \varepsilon^{\theta} \norm{ Iu_{1<\cdot\le N} }_{L^8_t L^\infty_x}^{1-\theta}. 
\end{align}
To estimate the second factor, we decompose dyadically in frequency space and use Bernstein, 
\begin{align}
    \| Iu_{1<\cdot\le N} \|_{L^8_t L^\infty_x}
    &\le \sum_{1< N_k\le N} \| P_{N_k} Iu \|_{L^8_t L^\infty_x} \\
    &\lesssim \sum_{1< N_k\le N} \| P_{N_k} N_k^{1/4} Iu \|_{L^8_t L^4_x} \\
    &\lesssim \sum_{1< N_k\le N} N_k^{-3/4} \| P_{N_k} \nabla Iu \|_{L^8_t L^4_x} \\
    &\lesssim \| \nabla Iu \|_{L^8_t L^4_x} \\
    &\le \| \nabla Iu \|_{Y(J)}.
\end{align}
For the high frequencies, using Bernstein,
\begin{align}
    \| u_{>N} \|_{L^8_t L^{12}_x}
    &\le \sum_{N_k>N} \frac{N_k^{-s}}{N^{1-s}} \| I\nabla u_{N_k} \|_{L^8_t L^{12}_x} \\
    &\lesssim \sum_{N_k>N} \frac{N_k^{\frac{1}{6}-s}}{N^{1-s}}
       \| I\nabla u_{N_k} \|_{L^8_t L^4_x} \\
    &\lesssim \| \nabla Iu \|_{S^0} \frac{N^{1/6 - s}}{N^{1-s}} \\
    &\lesssim \| \nabla Iu \|_{Y(J)} N^{-5/6}
\end{align}
for $s>1/6$. Hence, we find that
\begin{align}
    \norm{\nabla I u}_{Y(J)} \lesssim \norm{\nabla Iu_0}_{L^2} + \norm{\nabla I u}_{Y(J)} \e^6 + \e^{6\theta} \norm{\nabla I u}_{Y(J)}^{7-6\theta} + N^{-5} \norm{\nabla I u}_{Y(J)}^7.
\end{align}
The result follows from a bootstrap argument for small $\e$ and large $N$ since $\norm{\nabla Iu_0}_{L^2} \lesssim 1$.
\end{proof}
\begin{remark}
     A similar argument shows that under the hypotheses of Lemma \Ref{lem LWP}, we have that
    \begin{align}
        \norm{\nabla I u}_{S^0_*} &\lesssim \norm{\nabla I u_0}_{L^2_x}, \\
        \norm{\langle \nabla\rangle  I u}_{S^0} &\lesssim \norm{\langle \nabla \rangle I u_0}_{L^2_x}
    \end{align}
\end{remark}

\section{Energy Increment}\label{Energy Increment}
We now control the energy increment from above on intervals $J$ where the modified Morawetz norm $\norm{Iu}_{L^{8}_{t}L^{10}_x (J\times \R)}$ is small.
\begin{theorem}\label{Energy Increment Theorem}
    If $J$ is an interval such that $\sup_{t\in J} E(Iu(t))\le 1 $, $\norm{Iu}_{L^{8}_{t}L^{10}_x (J\times \R)} \le \e$,  and $s>2/5$, then for large enough $N$ and small enough $\e$, 
    \begin{align}
        \sup_{t_1, t_2\in J} |E(Iu(t_2))-E(Iu(t_1))| \lesssim  (1+\norm{u_0}_{L^2_x(\R)}) N^{-2+}. 
    \end{align}
\end{theorem}
Before proceeding with the proof of Theorem \ref{Energy Increment Theorem}, we first state and prove a multilinear estimate that will be essential to its proof. 
\begin{proposition} \label{multilinear prop}
    Consider a finite, even collection of functions $\{ f_j(t,x)\}_{j=1}^{m}$ defined on the slab $J\times \R$, where $f_j(t, \cdot)\in \mathcal{S}(\R)$ for each $t\in J$ and for each $j$, $\widehat{f_j}$ is supported in the annulus $\{\xi: N_j/2\le |\xi|\le 2N_j\}$. Denote $\mathbf{N}:=(N_1, \cdots, N_m)$ and let $\sigma_{\mathbf{N}}(\xi_1, \cdots, \xi_{m})$ be a smooth symbol satisfying
    \begin{align}  \label{symbol condition}
    \sup_{(\xi_1, \cdots, \xi_m)\in A_{\mathbf{N}}} |(N_1\partial_{\xi_1})^{\alpha_1}\cdots(N_m\partial_{\xi_m})^{\alpha_m}  \sigma_{\mathbf{N}}(\xi_1, \cdots, \xi_m)| \lesssim_{\alpha} C_{\mathbf{N}}
    \end{align}
for any multi-index $\alpha=(\alpha_1, \cdots, \alpha_m)$ such that $0\le \alpha_j \le 2$, where
\begin{align}
    A_{\mathbf{N}}: =  \prod_{j=1}^m \{\xi\in \mathbb{R}: N_j/4 \le |\xi| \le 4N_j\}
\end{align}
is an anisotropic annulus. Then, we have the estimate
\begin{align} \label{Mutlilinear estimate}
    &\left| \int_J \int_{\xi_1+\cdots +\xi_m=0} \sigma_{\mathbf{N}}(\xi_1, \cdots, \xi_m) \widehat{\overline{f_1}}(t,\xi_1)\,\widehat{f_2}(t,\xi_2) \ \cdots \widehat{\overline{f_{m-1}}}(t,\xi_{m-1})\,\widehat{f_m} (t,\xi_m)\,d\xi_1\cdots d\xi_m\,dt \right| \\
    &\lesssim C_{\mathbf{N}} \sup_{(x_1, \cdots, x_m)\in \R^m} \left| \int_J\int_{\R} \overline{T_{x_1}f_1(t,x)} T_{x_2}f_2(t,x)  \cdots \overline{T_{x_{m-1}}f_{m-1}(t,x)} T_{x_m}f_m(t,x)\,dx\,dt \right|,
\end{align}
where $T_{x_j}f:= f(\cdot+x_j)$ and the implicit constant depends on $\sigma$.
\end{proposition}
\begin{remark}
    Proposition \ref{multilinear prop} can be adapted to higher dimensions as well, as long as there is enough regularity on the symbol $\sigma$.
\end{remark}
\begin{remark}
    A result like Proposition \ref{multilinear prop} was used implicitly in \cite{Dod12} when bounding the energy-increment via the bilinear estimates but we include it here to make the proof self-contained.
\end{remark}
\begin{proof}[Proof of Proposition \ref{multilinear prop}]
Denote the left hand side of \eqref{Mutlilinear estimate} by $M_{\mathbf{N}}$. We may, without changing its value, freely replace the multiplier $\sigma_{\mathbf{N}}(\xi_1, \cdots, \xi_m)$ by 
\begin{align}
    \widetilde{\sigma}(\xi_1, \cdots, \xi_m):=\psi_{N_1}(\xi_1)\cdots \psi_{N_m}(\xi_m) \sigma_{\mathbf{N}}(\xi_1, \cdots, \xi_m),
\end{align}
where $\psi_{N_j}(\xi_j)=\psi(\xi_j/N_j)$ with $\psi\in C_c^\infty(\R)$ such that $\psi(\xi)=1$ for $\frac{1}{2} \le |\xi| \le 2$ and supported in $\{ \frac{1}{4} \le |\xi| \le 4\}$. Note that $\widetilde{\sigma}\in C_c^\infty(\R^m)$ and therefore we may write (up to harmless factors of $2\pi$)
\begin{align}
    \widetilde{\sigma} = \int_{\R^m} e^{ix_1\xi_1+\cdots + ix_m\xi_m} \widehat{\widetilde{\sigma}}(x_1, \ldots,x_m) dx_1\ldots dx_m
\end{align}
with $\widehat{\widetilde{\sigma}}$ rapidly decaying and smooth. We now proceed to obtain decay bounds on $\widehat{\widetilde{\sigma}}$.
Note that by Fourier inversion and integration by parts in each of the individual $\xi_j$,
\begin{align}
    |\widehat{\widetilde{\sigma}}| &= \left|\int_{\R^m} e^{-i(x_1\xi_1+\cdots + x_m\xi_m)} \widetilde{\sigma}(\xi_1, \ldots,\xi_m) d\xi_1\ldots d\xi_m\right| \\
    &\le \left(\prod_{j=1}^m \frac{1}{1+N_j^2 x_j^2}\right) \int_{\R^m} |(1-N_1^2\partial_{\xi_1}^2) \cdots (1-N_m^2\partial_{\xi_m}^2)\widetilde{\sigma}|d\xi_1\ldots d\xi_m.
\end{align}
Since each $\psi_{N_j}(\xi_j)$ is smooth on scale $N_j$, we have that, by the Leibniz rule and our bound on $\sigma_{\mathbf{N}}$, 
\begin{align}\sup_{(\xi_1,\cdots, \xi_m)\in A_{\mathbf{N}}} \left| \left(\prod_{j=1}^m (1-N_j^2\partial_{\xi_j}^2)\right)\widetilde{\sigma}\right| \lesssim C_{\mathbf{N}}.
\end{align}
Furthermore, due to the compact support of $\widetilde{\sigma}$, the integral above can be taken on the set $\{|\xi_j|\sim N_j, j=1, \cdots, m\}$. Therefore,
\begin{align}
    |\widehat{\widetilde{\sigma}}| \lesssim C_{\mathbf{N}}\prod_{j=1}^m \frac{N_j}{1+N_j^2 x_j^2}.
\end{align}
From this, we deduce that
\begin{align}
    \| \widehat{\widetilde{\sigma}}\|_{L^1(\R^m)} \lesssim_m C_{\mathbf{N}}.
\end{align}
Hence, we may write, with $T_{x_j}g= g(\cdot+x_j)$, the left hand side of \eqref{Mutlilinear estimate} as follows (up to taking absolute values),
\begin{align}
    M_{\mathbf{N}}=& \int_J \int_{\sum_{j=1}^m \xi_j=0}  \widetilde{\sigma}(\xi_1, \dots, \xi_m) \widehat{\overline{f_1}}(t, \xi_1)  \widehat{f_2}(t, \xi_2) \cdots  \\
    &\times \widehat{\overline{f_{m-1}}}(t, \xi_{m-1})\widehat{f_{m}}(t, \xi_m) \, d\xi_1\cdots d\xi_m\,dt \vphantom{\int_J} \\
    &= \int_{\R^m} \int_J  \int_{\sum_{j=1}^m \xi_j=0}  \widehat{\widetilde{\sigma}}(x_1, \ldots,x_m) e^{ix_1\xi_1}\widehat{\overline{f_1}}(t, \xi_1)  e^{ix_2\xi_2}\widehat{f_2}(t, \xi_2)  \cdots  \\
    &\times e^{ix_{m-1}\xi_{m-1}}\widehat{\overline{f_{m-1}}}(t, \xi_{m-1}) e^{ix_m\xi_m}\widehat{f_m}(t, \xi_m) \, d\xi_1\cdots d\xi_m\,\,dx_1\cdots dx_m \vphantom{\int_J} \\
    &= \int_{\R^m} \int_J  \int_{\sum_{j=1}^m \xi_j=0}  \widehat{\widetilde{\sigma}}(x_1, \ldots,x_m) \widehat{\overline{T_{x_1}f_1}}(t, \xi_1)  \widehat{T_{x_2}f_2}(t, \xi_2) \cdots  \\
    &\times \widehat{\overline{T_{x_{m-1}}f_{m-1}}}(t, \xi_{m-1}) \widehat{T_{x_m}f_m}(t, \xi_m) \, d\xi_1\cdots d\xi_m\,\,dx_1\cdots dx_m \vphantom{\int_J}.
\end{align}
Therefore, using the decay bounds on $\widehat{\widetilde{\sigma}}$ and the fact that
\begin{align}
    \int_{\xi_1+\cdots +\xi_m=0} \widehat{\overline{f_1}}(\xi_1)\,\widehat{f_2}(\xi_2) \,\widehat{\overline{f_3}}(\xi_3) \cdots \widehat{\overline{f_{m-1}}}(\xi_{m-1})\,\widehat{f_m} (\xi_m)\,d\xi_1\cdots d\xi_m = \int_{\R} \overline{f_1(x)} f_2(x) \overline{f_3(x)} \cdots \overline{f_{m-1}(x)} f_m(x)\,dx,
\end{align}
we conclude
\begin{align}
    |M_{\mathbf{N}}| 
    &\le 
    \| \widehat{\widetilde{\sigma}}\|_{L^1_x(\R^m)}
    \sup_{x_1,\cdots,x_m\in \R}  
    \bigg| 
    \int_J  
    \int_{\sum_{j=1}^m \xi_j=0}  
    \widehat{\overline{T_{x_1}f_1}}(t,\xi_1)
    \widehat{T_{x_2}f_2}(t,\xi_2)
    \cdots  
    \widehat{\overline{T_{x_{m-1}}f_{m-1}}}(t,\xi_{m-1})
    \widehat{T_{x_m}f_m}(t,\xi_m)
    \, d\xi\,dt\,
    \bigg|  \\
    &= 
    \| \widehat{\widetilde{\sigma}}\|_{L^1_x(\R^m)}
    \sup_{x_1,\cdots,x_m\in \R}  
    \bigg|
    \int_J  
    \int_{\R}  
    \overline{T_{x_1}f_1}(t,y)
    T_{x_2}f_2(t,y)
    \cdots  
    \overline{T_{x_{m-1}}f_{m-1}}(t,y)
    T_{x_m}f_m(t,y)
    \, dy\,dt\,
    \bigg| \\
    &\lesssim C_{\mathbf{N}}
    \sup_{x_1,\cdots,x_m\in \R}  
    \bigg|
    \int_J  
    \int_{\R}  
    \overline{T_{x_1}f_1}(t,y)
    T_{x_2}f_2(t,y)
    \cdots  
    \overline{T_{x_{m-1}}f_{m-1}}(t,y)
    T_{x_m}f_m(t,y)
    \, dy\,dt\,
    \bigg|.
\end{align}
as desired.
\end{proof}
We are now ready to proceed with the proof of Theorem \ref{Energy Increment Theorem}.\par
\begin{proof}[Proof of Theorem \ref{Energy Increment Theorem}] 
Take $\e>0$ small enough such that Lemma \ref{lem LWP} is satisfied. Hence, we have that
\begin{align}
    \|\nabla Iu\|_{Y(J)} \lesssim 1.
\end{align}
In analogy with the calculation
\begin{align}
    \frac{d}{dt} E(u)= \operatorname{Re}\int \overline{\partial_t u} (u|u|^6 -\Delta u)\,dx =  \operatorname{Re}\int \overline{u_t} (u|u|^6 -\Delta u-i\partial_t u)\,dx, \quad u(t)\in \mathcal{S}(\R),
\end{align}
we calculate that for $u(t)$ a Schwartz-class solution of \eqref{septic NLS},
\begin{align}
    \frac{d}{dt} E(Iu) &=  \operatorname{Re}\int \overline{I\partial_t u}(Iu|Iu|^6 -\Delta Iu - i\partial_t Iu)\,dx \\
    &= \operatorname{Re}\int \overline{I\partial_t u}(Iu|Iu|^6 - I(u|u|^6))\,dx.
\end{align}
Hence, by Plancherel and integration in time over $J$, 
\begin{align}
     |E(Iu(t_2))-E(Iu(t_1))| \le K+P,
\end{align}
where 
\begin{align*}
K &:= \\
&\int_J \left| \int_{\sum_{j=1}^8 \xi_j=0} |\xi_1|^2 \widehat{\overline{Iu}}(t, \xi_1)\left[1-\frac{m(\xi_2 + \cdots + \xi_8)}{m(\xi_2) \cdots m(\xi_8)}\right] \widehat{Iu}(t, \xi_2)\widehat{\overline{Iu}}(t, \xi_3) \cdots  
    \times \widehat{\overline{Iu}}(t, \xi_7)\widehat{Iu}(t, \xi_8) \, d\xi\right|\,dt    
\end{align*}
and 
\begin{align}\label{Potential}
    P&:= \int_J \left|\operatorname{Re}\int_{\sum_{j=1}^8 \xi_j=0}  i\widehat{\overline{I(u|u|^6)}}(t, \xi_1)\left[1-\frac{m(\xi_2 + \cdots + \xi_8)}{m(\xi_2) \cdots m(\xi_8)}\right] \widehat{Iu}(t, \xi_2)\widehat{\overline{Iu}}(t, \xi_3) \cdots \times 
    \widehat{\overline{Iu}}(t, \xi_7)\widehat{Iu}(t, \xi_8) \, d\xi\,\right|dt.
\end{align}

We will estimate $K$ and $P$ using a Littlewood--Paley decomposition, breaking into a number of cases and estimating each case individually. To this end, let 
\begin{align}
    u= P_1u + \sum_{N\in 2^{\Z}, N > 1} P_N u,
\end{align}
where $P_1=P_{\le 1}$ includes all frequencies smaller than $1$. 
We notice that by symmetry, and with $N_j$ dyadic,
\begin{align}    K  &\lesssim \sum_{N_1}\sum_{N_2\ge N_3 \ge \ldots \ge N_8} \int_J\left| \int_{\sum_{j=1}^8 \xi_j=0} |\xi_1|^2 \widehat{\overline{Iu_{N_1}}}(t, \xi_1)\left[1-\frac{m(\xi_2 + \cdots + \xi_8)}{m(\xi_2) \cdots m(\xi_8)}\right] \widehat{Iu_{N_2}}(t, \xi_2)\widehat{\overline{Iu_{N_3}}}(t, \xi_3) \cdots \right. \\
    &\times \left. \widehat{\overline{Iu_{N_7}}}(t, \xi_7)\widehat{Iu_{N_8}}(t, \xi_8) \, d\xi \vphantom{\int_J} \right|\,dt.
\end{align}
Define
\begin{align} \label{Kinetic localized}
    K_{N_1, \cdots, N_8} &:= \int_J \int_{\sum_{j=1}^8 \xi_j=0} |\xi_1|^2 \widehat{\overline{Iu_{N_1}}}(t, \xi_1)\left[1-\frac{m(\xi_2 + \cdots + \xi_8)}{m(\xi_2) \cdots m(\xi_8)}\right] \widehat{Iu_{N_2}}(t, \xi_2)\widehat{\overline{Iu_{N_3}}}(t, \xi_3) \cdots  \\
    &\times \widehat{\overline{Iu_{N_7}}}(t, \xi_7)\widehat{Iu_{N_8}}(t, \xi_8) \, d\xi\,dt \vphantom{\int_J}.
\end{align}
Define
\begin{align}
    \sigma_{\mathbf{N}}(\xi_1, \dots, \xi_8):= 1-\frac{m_N(\xi_2 + \dots + \xi_8)}{m_N(\xi_2) \cdots m_N(\xi_8)},
\end{align}
Note that due to the uniform control on $E(Iu(t))$ on the interval $J$ and the Gagliardo--Nirenberg inequality, we have that
\begin{align}
    \|Iu\|_{L^\infty_{t,x}(J\times \R)} \le \|\nabla Iu\|_{L^\infty_{t} L^2_x (J\times \R)}^{1/5} \| Iu\|_{L^\infty_{t} L^8_x (J\times \R)}^{4/5} \lesssim 1. 
\end{align}
This will prove quite useful in the subsequent analysis and we will use it repeatedly. 

\textbf{The term $K$}
\newline
\textbf{Case $K$--$I$}: $N_j>1$  for each $j$ \newline
\textbf{Case $K$--$I_a$}: $N_2\ll N$. This case contributes zero as
\begin{align}
    1-\frac{m(\xi_2 + \cdots + \xi_8)}{m(\xi_2) \cdots m(\xi_8)}=0.
\end{align}

\textbf{Case $K$--$I_b$}: $ N_2 \gtrsim N \gg N_3$. 
In this case, we follow the ideas of \cite{Dod12} and use bilinear estimates to keep track of the lack of interactions between high and low frequencies. We will in particular use Proposition \ref{multilinear prop} to estimate $K_{N_1\cdots N_8}$, in which we will multiply and divide by $\xi_3$ in its integrand so that $Iu_{N_3}$ acquires a derivative.  \par 
We must have $N_1\sim N_2$. In this case, the effective multiplier satisfies, with $\eta=\xi_3+\cdots + \xi_8$, 
\begin{align} \label{FTC on multiplier}
    \frac{\sigma_{\mathbf{N}}(\xi)}{\xi_3} = \frac{1}{\xi_3}\left(1-\frac{m_N(\xi_2+\xi_3+ \cdots +\xi_8)}{m_N(\xi_2)}\right) = -\eta\int_0^1 \frac{m_N'(\xi_2+ t\eta)}{\xi_3 m_N(\xi_2)}\,dt.
\end{align}
Now, note that $m_N(\cdot), m_N(\cdot)^{-1}$ satisfy the symbol-type bounds
\begin{align}
    |\partial_\lambda^{\ell} m_N(\lambda)|\lesssim_\ell m_N(\lambda) M^{-\ell}, \quad |\lambda| \sim M \\ 
    |\partial_\lambda^{\ell} m_N(\lambda)^{-1}|\lesssim_\ell m_N(\lambda)^{-1} M^{-\ell}, \quad |\lambda| \sim M
\end{align}
for arbitrary $\ell\in \N_0$ as can be seen by considering the regions $M\ll N, M\sim N, M\gg N$ separately. Furthermore, we have that
\begin{align}
    |(N_j\partial_{\xi_j})^\ell \eta | \lesssim N_3 \quad 3\le j\le 8
\end{align}
for any $\ell\in \N_0$ since $N_3\gtrsim N_j, j\ge 4$ (in fact the left hand side vanishes for $\ell\ge 2$). 
Hence, by the Leibniz rule, the two above estimates with $M=N_2$, applied to \eqref{FTC on multiplier}, and the fact that $N_2\ge N_j, j\ge 3$,  we may conclude that $\sigma_{\mathbf{N}}(\xi)/\xi_3 $ satisfies \eqref{symbol condition} with $C_\mathbf{N}:=\frac{1}{N_2}$. 

Hence, we may apply Proposition \ref{multilinear prop} to conclude that by Hölder and Lemma \ref{lem bilinear} (noting that \eqref{septic NLS} is translation invariant),
\begin{align}
    K_{N_1\cdots N_8} &\lesssim \frac{1}{N_2} \sup_{x_1,\cdots,x_m\in \R}  
    \bigg|
    \int_J  
    \int_{\R}  
    \overline{T_{x_1}\Delta Iu_{N_1}}(t,y)\dots 
    \overline{T_{x_3}\nabla Iu_{N_3}}(t,y)
    \cdots  
    \overline{T_{x_{m-1}}Iu_{N_{m-1}}}(t,y)
    T_{x_m}Iu_{N_m}(t,y)
    \, dy\,dt \bigg| \\
    &\lesssim \frac{1}{N_2}\frac{N_1}{N_2N_4} \sup_{x_1,\cdots,x_4\in \R}  
    \norm{ T_{x_1}\nabla Iu_{N_1} \,T_{x_3} \nabla Iu_{N_3}}_{2,2} \norm{T_{x_2}\nabla Iu_{N_2} \,T_{x_4} \nabla Iu_{N_4}}_{2,2} \prod_{j=5}^8 \norm{Iu_{N_j}}_{\infty, \infty} \\
    &\lesssim \frac{N_1  }{N_2^2 N_4} \frac{1}{N_1^{1/2}} \frac{1}{N_2^{1/2}}\norm{\nabla Iu_0}_{L^2}^4 \\
    &\lesssim \frac{N_1 }{N_2^2 N_4} \frac{1}{N_1^{1/2}} \frac{1}{N_2^{1/2}}
\end{align}
The contribution of this case is therefore by Bernstein,
\begin{align} \label{Case I_b final estimate}
    \lesssim \sum_{N_1\sim N_2 \gtrsim N } \sum_{N_3\ge \cdots \ge N_8} \frac{N_1}{N_2^2} \frac{1}{N_4} \frac{1}{N_1^{1/2}} \frac{1}{N_2^{1/2}}
    \lesssim \sum_{N_1\sim N_2 \gtrsim N } \sum_{N_3\ge \cdots \ge N_8} \frac{N_1^{1/2}}{N_2^{5/2}} \frac{1}{N_4}  \lesssim N^{-2+}
\end{align}
\begin{remark}
For the remaining cases, we will be using Proposition \ref{multilinear prop} to reduce the estimate of $K_{N_1\dots N_8}$ to $L^1_{t,x}$ expressions but will no longer make this step explicit. Note that when $N_j=1$ for some $j$, Proposition \ref{multilinear prop} will need to be slightly modified but will apply in essentially the same way.
\end{remark}
\textbf{Case $K$--$I_c$}: $ N_2\ge N_3 \gtrsim N \gg N_4$. 
In this case, there is no cancellation to take advantage of in the multiplier. 
\begin{align}
    \left|1-\frac{m(\xi_1)}{m(\xi_2)m(\xi_3)}\right| \lesssim \frac{m(\xi_1)}{m(\xi_2)m(\xi_3)}
\end{align}

\textbf{Case $K$--$I_c(\alpha)$}: $ N_1\sim N_2 \gg N_3$
The contribution of this case is, using the same multilinear estimates as in case $I_b$, and Hölder's inequality,
\begin{align}
    &\lesssim \sum_{N_1\sim N_2 \gtrsim N } \sum_{N_3\gtrsim N} \frac{N_1}{N_2}\frac{1}{N_3m(N_3)} \norm{P_{N_1} \nabla Iu \,P_{N_3} \nabla Iu}_{2,2} \norm{P_{N_2} \nabla Iu \,P_{N_4} \nabla Iu}_{2,2}\frac{1}{N_4 } \prod_{j=5}^8 \norm{P_{N_j} Iu}_{\infty, \infty} \\
    &\lesssim \sum_{N_1\sim N_2 \gtrsim N } \sum_{N_3 \gtrsim N} \frac{N_1^{1/2}}{N_2^{3/2}} \frac{1}{N_3^s N^{1-s}} \frac{1}{N_4} \\
    &\lesssim N^{-2}.
\end{align}
\textbf{Case $K$--$I_c(\beta)$}: $ N_1 \ll  N_2 \sim  N_3$
The estimate in this case is
\begin{align}
    &\lesssim \sum_{N_2\sim N_3\gtrsim N } \sum_{N_1\ll  N_2} \frac{N_1m(N_1)}{N_2m(N_2) N_3m(N_3)} \norm{P_{N_1} \nabla Iu \,P_{N_3} \nabla Iu}_{2,2} \norm{P_{N_2} \nabla Iu \,P_{N_4} \nabla Iu}_{2,2}\frac{1}{N_4 } \prod_{j=5}^8 \norm{P_{N_j} Iu}_{\infty, \infty} \\
    &\lesssim \sum_{N_2\sim N_3\gtrsim N } \sum_{N_1\ll  N_2} \frac{N_1^s}{N_2^{s+1/2} N^{1-s} N_3^{s+1/2}} \\
    &\lesssim \sum_{N_2\sim N_3\gtrsim N }\frac{1}{N_2^{1/2} N^{1-s} N_3^{s+1/2}} \\
    &\lesssim N^{-2}
\end{align}
\textbf{Case $K$--$I_c(\gamma )$}: $N_1\sim N_2\sim N_3$.
In this case, 
\begin{align}
    \frac{m(\xi_1)}{m(\xi_2)m(\xi_3)} \sim \frac{1}{m(\xi_3)}.
\end{align}
Here, we will need to distinguish between $u_i=u$ and $u_i=\overline{u}$.
Since $\xi_1+ \cdots +\xi_8=0$, we estimate
\begin{align}
    \sum_{N_1\sim N_2\sim N_3\gtrsim N} \sum_{N_8\le \cdots \le N_4\ll N} \frac{N_1}{N_2 m(N_2)}  \norm{P_{\le 10N_4} \, (P_{N_1} \nabla \overline{Iu}) \,(P_{N_2} \nabla Iu_2) (P_{N_3} Iu_3) \,(P_{N_4} Iu_4)}_{1,1}\prod_{j=5}^8 \norm{P_{N_j} Iu}_{\infty, \infty}.
\end{align}
Since $N_3\gg N_4$ and letting $u_2=u$,
\begin{align}
    & \norm{P_{\le 10N_4} \, (P_{N_1} \nabla \overline{Iu}) \,(P_{N_2} \nabla Iu) (P_{N_3}  Iu_3) \,(P_{N_4}  Iu_4)}_{1,1} \\
    &\le \frac{1}{N_3 N_4} \norm{P_{N_3} (P_{N_1} \nabla \overline{Iu}) \,(P_{N_2} \nabla Iu)}_{2,2} \norm{(P_{N_3} \nabla Iu_3) \,(P_{N_4}  \nabla Iu_4)}_{2,2} \\
    &\lesssim \frac{1}{N_4 N_3^2}
\end{align}

The same calculation could be made if $u_2=\overline{u}$ and $u_3=u$. It therefore only remains to estimate
\begin{align}
    \norm{P_{\le 10N_4} \, (P_{N_1} \nabla \overline{Iu}) \,(P_{N_2} \nabla \overline{Iu}) (P_{N_3}  \overline{Iu}) \,(P_{N_4}  Iu_4)}_{1,1}.
\end{align}
Split $u=u_+ + u_{-}$, where $u_+$ is supported on $\xi>0$ and $u_-$ supported on $\xi<0$. Since $N_4 \ll N_1$, 
\begin{align}
    P_{\le 10N_4} \, (P_{N_1} \nabla Iu_+) \,(P_{N_2} \nabla Iu_+) (P_{N_3}  Iu_+) \,(P_{N_4}  \overline{Iu_4}) =0
\end{align}
and 
\begin{align}
    P_{\le 10N_4} \, (P_{N_1} \nabla Iu_-) \,(P_{N_2} \nabla Iu_-) (P_{N_3}  Iu_-) \,(P_{N_4}  \overline{Iu_4}) =0.
\end{align}
Since $P_{N_1} \nabla Iu_+$ and $P_{N_2} \nabla Iu_-$ are separated by $\sim N_3$, we have
\begin{align}
    & \norm{P_{\le 10N_4} \, (P_{N_1} \nabla Iu_+) \,(P_{N_2} \nabla Iu_-) (P_{N_3}  Iu_+) \,(P_{N_4}  \overline{Iu_4})}_{1,1} \\
    &\le \norm{(P_{N_1} \nabla Iu_+) \,(P_{N_2} \nabla Iu_-)}_{2,2} \norm{(P_{N_3}  Iu_+) \,(P_{N_4}  \overline{Iu_4})}_{2,2}  \\
    &\lesssim \frac{1}{N_4 N_3^2}.
\end{align}
Therefore, this case contributes in total
\begin{align}
    &\lesssim \sum_{N_1\sim N_2\sim N_3\gtrsim N} \sum_{N_8\le \cdots \le N_4\ll N} \frac{N_1}{N_2 m(N_2)}  \frac{1}{N_4N_3^2} \\
    &\lesssim N^{-2}
\end{align}
\textbf{Case $K$--$I_d$}: $ N_2\ge N_3 \ge N_4 \gtrsim N \gg N_5$. \newline
The multiplier in this case obeys
\begin{align}
    \left| 1- \frac{m(\xi_1)}{m(\xi_2) \cdots m(\xi_4)}\right| \lesssim \frac{m(N_1)}{m(N_2)m(N_3)m(N_4)}.
\end{align}
\textbf{Case $K$--$I_d(\alpha)$}: $N_1\sim N_2$. \newline 
Thus, the contribution for this case is, by Hölder and Bernstein,
\begin{align}
    &\lesssim \sum_{N_1\sim N_2\gtrsim N}\frac{N_1}{N_2} \norm{   (P_{N_1}\, N_1 I u)}_{6,6} \norm{(P_{N_2} N_2 I u)}_{6,6} \sum_{N_3\ge N_4\gtrsim N} \sum_{N_5 \cdots, N_8 \ll N} \norm{P_{N_3} \nabla I u}_{6,6}\norm{P_{N_4} \nabla I u}_{6,6} \\
    &\times \frac{1}{N_3m(N_3)}\frac{1}{N_4m(N_4)} \prod_{j=5}^8 \norm{P_{N_j} Iu}_{12,12} \\
    &\lesssim N^{-2} \left(\sum_{N_1 \sim N_2\gtrsim N} \frac{N_1}{N_2} \norm{P_{N_1}\, N_1 I u}_{6,6}^2 \right)^{1/2} \left(\sum_{N_1 \sim N_2\gtrsim N} \frac{N_1}{N_2} \norm{P_{N_2}\, N_2 I u}_{6,6}^2 \right)^{1/2} 
    \end{align}
where we used Cauchy--Schwarz in $N_1, N_2$. Now we can sum over $N_1, N_2$ to obtain that the above product of sums is bounded by  
\begin{align}
    \lesssim &\left(\sum_{N_1 \gtrsim N}  \norm{P_{N_1}\, N_1 I u}_{6,6}^2 \right)^{1/2} \left(\sum_{ N_2\gtrsim N}  \norm{P_{N_2}\, N_2 I u}_{6,6}^2 \right)^{1/2}  \\
    &\lesssim \norm{\nabla I u}_{Z(J)}^2 \\
    &\lesssim 1
    \end{align}
by the definition of the $Z(J)$ norm. Going forward, we will freely use the estimate
\begin{align}
    \sum_{N_1\sim N_2\gtrsim N}\frac{N_1}{N_2} \norm{   (P_{N_1}\, N_1 I u)}_{6,6} \norm{(P_{N_2} N_2 I u)}_{6,6} \lesssim \norm{\nabla I u}_{Z(J)}^2 \lesssim 1
\end{align}
\newline 
\textbf{Case $K$--$I_d(\beta)$}: $N_1 \ll N_2$.
In this case, we must have $N_2 \sim N_3$ since if not, $N_2 \gg N_1+N_3+ \cdots +N_8$, and this contradicts $\xi_1+\cdots + \xi_8=0.$
The contribution for this case is 
\begin{align}
    &\lesssim \sum_{N_1\ll N_2 \sim N_3} \sum_{ N_4\gtrsim N} \sum_{N_5 \cdots, N_8 \ll N}\frac{m(N_1) N_1}{m(N_2) N_2}  \norm{P_{N_1}\, \nabla I u}_{6,6}\norm{P_{N_2} \nabla I u}_{6,6}  \norm{P_{N_3} \nabla I u}_{6,6}\norm{P_{N_4} \nabla I u}_{6,6} \\
    &\times \frac{1}{N_3m(N_3)}\frac{1}{N_4m(N_4)} \prod_{j=5}^8 \norm{P_{N_j} Iu}_{12,12} \\
    &\lesssim \sum_{N_2, N_4\gtrsim N} \frac{N_2^{-s+}}{N^{1-s}} \frac{N_4^{-s}}{N^{1-s}} \\
    &\lesssim N^{-2+}
\end{align}
\textbf{Case $K$--$I_e$}: $ N_2\ge N_3 \ge N_4\ge N_5 \gtrsim N \gg N_6$. \newline
The multiplier is estimated as
\begin{align}
    \left| 1- \frac{m(\xi_1)}{m(\xi_2) \cdots m(\xi_5)}\right| \lesssim \frac{m(N_1)}{m(N_2)m(N_3)m(N_4)m(N_5)}.
\end{align}

\textbf{Case $K$--$I_e(\alpha)$}: $N_1\sim N_2$. The contribution in this case can be estimated by, using Hölder and Bernstein, 
\begin{align}
    &\lesssim \sum_{N_1\sim N_2\gtrsim N}\frac{N_1}{N_2} \norm{   (P_{N_1}\, N_1 I u)}_{6,6} \norm{(P_{N_2} N_2 I u)}_{6,6} \sum_{N_3\ge N_4\ge N_5\gtrsim N} \sum_{N_6,  \cdots, N_8 \ll N} \norm{P_{N_3} \nabla I u}_{6,6}\norm{P_{N_4} \nabla I u}_{6,6} \\
    &\times \norm{P_{N_5} \nabla I u}_{6,6}\frac{1}{N_3m(N_3)}\frac{1}{N_4m(N_4)} \frac{1}{N_5m(N_5)}\prod_{j=6}^8 \norm{P_{N_j} Iu}_{18,18} \\
    &\lesssim N^{-3}
\end{align}

\textbf{Case $K$--$I_e(\beta)$}: $N_1 \ll N_2$. We again must have $N_2\sim N_3$. \newline
The contribution in this case is, by Hölder and Bernstein,
\begin{align}
    &\lesssim \sum_{N_1\ll N_2 \sim N_3} \sum_{ N_4\ge N_5\gtrsim N} \sum_{N_6 \cdots, N_8 \ll N}\frac{m(N_1) N_1}{m(N_2) N_2}  \norm{P_{N_1}\, \nabla I u}_{6,6}\norm{P_{N_2} \nabla I u}_{6,6}  \norm{P_{N_3} \nabla I u}_{6,6}\norm{P_{N_4} \nabla I u}_{6,6} \\
    &\times \norm{P_{N_5} \nabla I u}_{6,6}\frac{1}{N_3m(N_3)}\frac{1}{N_4m(N_4)} \frac{1}{N_5m(N_5)}\prod_{j=6}^8 \norm{P_{N_j} Iu}_{18,18} \\
    &\lesssim  \sum_{N_2, N_4, N_5\gtrsim N} \frac{N_2^{-s+}}{N^{1-s}} \frac{N_4^{-s}}{N^{1-s}}\frac{N_5^{-s}}{N^{1-s}} \\
    &\lesssim N^{-3+}
\end{align}
\textbf{Case $K$--$I_f$}: $ N_2\ge N_3 \ge N_4\ge N_5 \ge N_6 \gtrsim N \gg N_7$. \newline
The multiplier is estimated as
\begin{align}
    \left| 1- \frac{m(\xi_1)}{m(\xi_2) \cdots m(\xi_8)}\right| \lesssim \frac{m(N_1)}{m(N_2)m(N_3)m(N_4)m(N_5)m(N_6)}.
\end{align}
\textbf{Case $K$--$I_f(\alpha)$}: $N_1\sim N_2$. The contribution in this case can be estimated by, using Hölder and Bernstein,
\begin{align}
    &\lesssim \sum_{N_1\sim N_2\gtrsim N}\frac{N_1}{N_2} \norm{   (P_{N_1}\, N_1 I u)}_{6,6} \norm{(P_{N_2} N_2 I u)}_{6,6} \sum_{N_3\ge N_4\ge N_5\ge N_6\gtrsim N} \sum_{N_7, N_8 \ll N} \norm{P_{N_3} \nabla I u}_{6,6}\norm{P_{N_4} \nabla I u}_{6,6} \\
    &\times \norm{P_{N_5} \nabla I u}_{6,6}\norm{P_{N_6} \nabla I u}_{6,6}\frac{1}{N_3m(N_3)}\frac{1}{N_4m(N_4)} \frac{1}{N_5m(N_5)}\frac{1}{N_6m(N_6)}\prod_{j=7}^8 \norm{P_{N_j} Iu}_{\infty, \infty} \\
    &\lesssim N^{-4}
\end{align}
\textbf{Case $K$--$I_f(\beta)$}: $N_1 \ll N_2$. We again must have $N_2\sim N_3$. \newline
The contribution in this case is by Hölder and Bernstein,
\begin{align}
    &\lesssim \sum_{N_1\ll N_2 \sim N_3} \sum_{ N_4\ge N_5\ge N_6\gtrsim N} \sum_{N_7, N_8 \ll N}\frac{m(N_1) N_1}{m(N_2) N_2}  \norm{P_{N_1}\, \nabla I u}_{6,6}\norm{P_{N_2} \nabla I u}_{6,6}  \norm{P_{N_3} \nabla I u}_{6,6}\norm{P_{N_4} \nabla I u}_{6,6} \\
    &\times \norm{P_{N_5} \nabla I u}_{6,6}\norm{P_{N_6} \nabla I u}_{6,6}\frac{1}{N_3m(N_3)}\frac{1}{N_4m(N_4)} \frac{1}{N_5m(N_5)}\frac{1}{N_6m(N_6)}\prod_{j=7}^8 \norm{P_{N_j} Iu}_{\infty, \infty} \\
    &\lesssim  \sum_{N_2, N_4, N_5, N_6\gtrsim N} \frac{N_2^{-s+}}{N^{1-s}} \frac{N_4^{-s}}{N^{1-s}}\frac{N_5^{-s}}{N^{1-s}} \frac{N_6^{-s}}{N^{1-s}}\\
    &\lesssim N^{-4+}
\end{align}
\textbf{Case $K$--$I_g$}: $ N_2\ge N_3 \ge N_4\ge N_5 \ge N_6 \ge N_7\gtrsim N \gg N_8$. \newline
The multiplier is estimated as
\begin{align}
    \left| 1- \frac{m(\xi_1)}{m(\xi_2) \cdots m(\xi_8)}\right| \lesssim \frac{m(N_1)}{m(N_2)m(N_3)m(N_4)m(N_5)m(N_6)m(N_7)}.
\end{align}
\textbf{Case $K$--$I_g(\alpha)$}: $N_1\sim N_2$.
The contribution is, by Hölder and Bernstein,
\begin{align}
    &\lesssim \sum_{N_1\sim N_2\gtrsim N}\frac{N_1}{N_2} \norm{   (P_{N_1}\, N_1 I u)}_{6,6} \norm{(P_{N_2} N_2 I u)}_{6,6} \sum_{N_3\ge N_4\ge N_5\ge N_6\ge N_7 \gtrsim N} \sum_{ N_8 \ll N} \norm{P_{N_8} Iu}_{\infty, \infty} \prod_{j=3}^7 \frac{1}{m(N_j)} \norm{ P_{N_j} I u}_{15/2, 15/2} \\
    &\lesssim \sum_{N_3\ge N_4\ge N_5\ge N_6\ge N_7 \gtrsim N} \prod_{j=3}^7 \frac{1}{m(N_j)} \norm{ |\nabla|^{1/10} P_{N_j} I u}_{15/2, 30/7} \\
    &\lesssim \sum_{N_3\ge N_4\ge N_5\ge N_6\ge N_7 \gtrsim N} \prod_{j=3}^7 \frac{1}{N_j^{9/10} m(N_j)} \lesssim \left(N^{-9/10}\right)^{5}\lesssim N^{-2}
\end{align}
for $s>1/10$. \newline
\textbf{Case $K$--$I_g(\beta)$}: $N_1 \ll N_2$. Again, we must have $N_2\sim N_3\gtrsim N$. \newline 
By Hölder and Bernstein, the contribution is
\begin{align}
    &\lesssim \sum_{N_1\ll N_2 \sim N_3} \sum_{ N_4\ge N_5\ge N_6\ge N_7 \gtrsim N} \sum_{ N_8 \ll N}\frac{m(N_1) N_1}{m(N_2) N_2}  \norm{P_{N_1}\, \nabla I u}_{6,6}\norm{P_{N_2} \nabla I u}_{6,6}  \prod_{j=3}^7 \frac{1}{m(N_j)} \norm{P_{N_j} I u}_{15/2,15/2} \\
    &\times \norm{P_{N_8} I u}_{\infty, \infty} \\
    &\lesssim \sum_{ N_2 \sim N_3\gtrsim N}  N_2^+ \sum_{ N_4\ge N_5\ge N_6\ge N_7 \gtrsim N} \prod_{j=3}^7 \frac{1}{N_j^{9/10} m(N_j)} \\
    &\lesssim N^{-9/2+}
\end{align}
for $s>1/10$. \newline 
\textbf{Case $K$--$I_h$}: $ N_2\ge N_3 \ge N_4\ge N_5 \ge N_6 \ge N_7\ge N_8\gtrsim N $. \newline
The multiplier is estimated as
\begin{align}
    \left| 1- \frac{m(\xi_1)}{m(\xi_2) \cdots m(\xi_8)}\right| \lesssim \frac{m(N_1)}{m(N_2)m(N_3)m(N_4)m(N_5)m(N_6)m(N_7)m(N_8)}.
\end{align}
\textbf{Case $K$--$I_h(\alpha)$}: $N_1\sim N_2$.
\newline
This case contributes, by Hölder and Bernstein, 
\begin{align}
    &\lesssim \sum_{N_1\sim N_2\gtrsim N}\frac{N_1}{N_2} \norm{   (P_{N_1}\, N_1 I u)}_{6,6} \norm{(P_{N_2} N_2 I u)}_{6,6} \sum_{N_3\ge N_4\ge N_5\ge N_6\ge N_7\ge N_8 \gtrsim N}   \prod_{j=3}^8 \frac{1}{m(N_j)} \norm{ P_{N_j} I u}_{9,9} \\
    &\lesssim \sum_{N_3\ge N_4\ge N_5\ge N_6\ge N_7 \ge N_8\gtrsim N} \prod_{j=3}^8 \frac{1}{m(N_j)} \norm{ |\nabla|^{1/6} P_{N_j} I u}_{9, 18/5} \\
    &\lesssim \sum_{N_3\ge N_4\ge N_5\ge N_6\ge N_7 \ge N_8 \gtrsim N} \prod_{j=3}^8 \frac{1}{N_j^{5/6} m(N_j)} \\
    &\lesssim \left(N^{-5/6}\right)^{6} \\
    &\lesssim N^{-5}
\end{align}
for $s>1/6$.\newline
\textbf{Case $K$--$I_h(\beta)$}: $N_1 \ll N_2$. Again, we must have $N_2\sim N_3\gtrsim N$. \newline 
The contribution is, by Hölder and Bernstein,
\begin{align}
    &\lesssim \sum_{N_1\ll N_2 \sim N_3} \sum_{ N_4\ge N_5\ge N_6\ge N_7 \ge N_8\gtrsim N}\frac{m(N_1) N_1}{m(N_2) N_2}  \norm{P_{N_1}\, \nabla I u}_{6,6}\norm{P_{N_2} \nabla I u}_{6,6}  \prod_{j=3}^8 \frac{1}{m(N_j)} \norm{P_{N_j} I u}_{9,9} \\
    &\lesssim \sum_{ N_2 \sim N_3\gtrsim N}  N_2^+ \sum_{ N_4\ge N_5\ge N_6\ge N_7\ge N_8 \gtrsim N} \prod_{j=3}^8 \frac{1}{N_j^{5/6} m(N_j)} \\
    &\lesssim N^{-5+}
\end{align}
for $s>1/6$.\par
We now move on to the cases where we have at least one frequency $N_j$ such that $N_j=1$ (recall that $P_1\equiv P_{\le 1}$). \par
\textbf{Case $K$--$II$}: there exists $j$ such that $N_j=1$. \par 
\textbf{Case $K$--$II_a$}: $N_1=1$ and there exists $L$ such that $N_2\ge \cdots \ge N_L>1=N_{L+1}\ge \cdots \ge N_8$. Note that we must have $L\ge 3$ so that $N_2\sim N_3\gtrsim N$.\newline 
The multiplier is estimated as
\begin{align}
    \left| 1- \frac{m(\xi_1)}{m(\xi_2) \cdots m(\xi_8)}\right| \lesssim \frac{m(N_1)}{m(N_2)\cdots m(N_L)}.
\end{align}
The contribution of this case is, using Hölder,
\begin{align}
    &\lesssim \sum_{N_2\sim N_3\gtrsim N} \sum_{N_4\ge \cdots \ge N_L >1} N_1^2 \frac{1}{N_2m(N_2)} \frac{1}{N_3 m(N_3)}\norm{P_{N_2}\nabla Iu}_{8,4}\norm{P_{N_3}\nabla Iu}_{8,4} \\
    &\times \prod_{j=4}^L \frac{1}{m(N_j)} \norm{P_{N_j} Iu}_{8,12} \times \prod_{j\in\{1, L+1, \cdots, 8\}} \norm{P_{N_j\le 1} Iu}_{8,12}.
\end{align}
To sum, we first notice that $N_1\le 1$. We estimate the frequencies $N_j, 4\le j\le L$ as follows:
\begin{align}
    \frac{1}{m(N_j)} \norm{P_{N_j} Iu}_{8,12}  &\lesssim \frac{1}{m(N_j)} \norm{|\nabla|^{1/6} P_{N_j}  Iu}_{8,4} \\
    &\lesssim \frac{1}{N_j^{5/6} m(N_j)} \norm{\nabla P_{N_j}  Iu}_{8,4} \\
    &\lesssim \frac{1}{N_j^{5/6} m(N_j)},
\end{align}
which is summable over $N_j>1$ if $s>1/6$. To estimate the low frequencies, we use Bernstein and our Morawetz bound to show that
\begin{align}
    \norm{P_{N_j\le 1} Iu}_{8,12} &\lesssim  \norm{P_{N_j\le 1} Iu}_{8,10} \\
    &\lesssim \e.
\end{align}
Lastly, we sum over $N_2, N_3$, and then over $L$ to give us a contribution
\begin{align*}
    \lesssim N^{-2} \sum_{L=3}^{8}\e^{9-L} \lesssim N^{-2}.
\end{align*}
\textbf{Case $K$--$II_b$}: $N_1>1$ and $N_2\gtrsim N \gg N_3$. \newline
\textbf{Case $K$--$II_b(\alpha)$}: $N_3=N_4=1$. Note that this implies that $N_1\sim N_2$. In this case, we use the fundamental theorem of calculus to see that 
\begin{align}
    \left(1- \frac{m(\xi_2+\cdots+\xi_8)}{m(\xi_2)}\right) =-\eta \int_0^1 \frac{m_N'(\xi_2+ t\eta)}{ m_N(\xi_2)}\,dt, 
\end{align}
where $\eta=\xi_3+\cdots +\xi_8$. Note that we will treat
\begin{align}
    \int_0^1 \frac{m_N'(\xi_2+ t\eta)}{ m_N(\xi_2)}\,dt
\end{align}
as our effective multiplier and turn $\eta$ into derivatives acting on the low frequency factors. Since these factors $P_{N_j}Iu, j=3,\dots, 8$ are interchangeable in this case, it is enough to treat the case when the derivative lands on $P_{N_3}Iu$, with the other cases estimated in the same way, except that now the factor with the derivative on it is estimated in the bilinear and the others go in $L^\infty_{t,x}$. The point is that this extra derivative guarantees that we do not acquire an extra copy of the mass $\|u_0\|_{L^2}$ in our estimates, which after rescaling, would make it less efficient.  \par 
The contribution, is by Hölder, Bernstein and the bilinear estimates,
\begin{align}
    &\sum_{N_1\sim N_2\gtrsim N} \frac{1}{N_2^2} N_1 \norm{P_{N_1}\nabla Iu P_{N_3}  \nabla Iu}_{2,2}\norm{P_{N_2}\nabla Iu P_{N_4}  Iu}_{2,2} \prod_{j=5}^8 \norm{P_{N_j\le 1} Iu}_{\infty, \infty} \\
    &\lesssim \sum_{N_1\sim N_2\gtrsim N} \norm{\nabla Iu_0}_{L^2_x}^3 \left(\norm{u_0}_{L^2}+ \norm{P_{\le 1} I(u|u|^6)}_{1,2}\right) \frac{N_1^{1/2}}{N_2^{5/2}}
\end{align}
From here, we estimate, using the fact that $\langle \nabla \rangle^{-1}$ is bounded on $L^p_x$ for $1\le p\le\infty$ (its symbol $\langle\xi\rangle^{-1}$ is of order $-1$),
\begin{align}
    \norm{ I(u|u|^6)}_{1,2} &\lesssim \norm{ \langle \nabla \rangle I( u|u|^6)}_{1,2} \\
    &\lesssim \norm{\langle \nabla \rangle Iu}_{4,\infty}\norm{|u|^6}_{4/3, 2} \\
    &\lesssim \norm{\langle \nabla \rangle I u}_{Y(J)} \left( \e^6 + \e^{6\theta} \norm{\nabla I u}_{Y(J)}^{6-6\theta} + N^{-5} \norm{\nabla I u}_{Y(J)}^6\right) \\
    &\lesssim 1+ \norm{u_0}_{L^2_x},
\end{align}
where we used an estimate in the proof of Lemma \Ref{lem LWP}. Hence, summing over $N_1, N_2$, we find that the contribution for this case is
\begin{align}
    \lesssim N^{-2} (1+\norm{u_0}_{L^2_x}).
\end{align}
\textbf{Case $K$--$II_b(\beta)$}: $N_3>1=N_4$. 
The multiplier estimate is the same as in the previous case. The contribution is thus by Hölder and Bernstein
\begin{align}
    &\sum_{N_1\sim N_2\gtrsim N} \sum_{1<N_3\ll N} \frac{1}{N_2^2} N_1 \|P_{N_1}\nabla Iu P_{N_3}  \nabla Iu\|_{2,2}\|P_{N_2}\nabla Iu P_{N_4}  Iu\|_{2,2} \prod_{j=5}^8 \|P_{N_j\le 1} Iu\|_{\infty, \infty} \\
    &\lesssim N^{0+}\sum_{N_1\sim N_2\gtrsim N} \|\nabla Iu_0\|_{L^2_x}^3 \left(\|u_0\|_{L^2}+ \|P_{\le 1} I(u|u|^6)\|_{1,2}\right) \frac{N_1^{1/2}}{N_2^{5/2}} \\
&\lesssim N^{-2+} (1+\|u_0\|_{L^2_x}).
\end{align}

\textbf{Case $K$--$II_b(\gamma)$}: $N_3\ge N_4\ge \cdots \ge N_L>1=N_{L+1}\ge \cdots \ge N_8$.
Note that we must have $L\ge 4$.
We further break this case into two scenarios: either $L=4$ or $L\ge 5$. In the latter case, we have, for $L$ fixed, by Hölder and Bernstein
\begin{align}
    &\sum_{N_1\sim N_2\gtrsim N} \sum_{1<N_L\le \cdots \le N_3\ll N} \frac{1}{N_2^2} N_1 \|P_{N_1}\nabla Iu P_{N_3}  \nabla Iu\|_{2,2}\|P_{N_2}\nabla Iu P_{N_4} Iu\|_{2,2} \prod_{j\in\{5, \cdots, L\}} \|P_{N_j} Iu\|_{\infty, \infty} \\
    &\times \prod_{j\in\{L+1, \cdots, 8\}}\|P_{N_j\le 1} Iu\|_{\infty, \infty}  \\
    &\lesssim \sum_{N_1\sim N_2\gtrsim N} \sum_{1<N_L\le \cdots \le N_3\ll N}\|\nabla Iu_0\|_{L^2_x}^3 \left(\|u_0\|_{L^2}+ \|P_{\le 1} I(u|u|^6)\|_{1,2}\right) \frac{N_1^{1/2}}{N_2^{5/2}}\prod_{j\in\{L+1, \cdots, 8\}} \|P_{N_j\le 1}Iu\|_{\infty, 8} \\
    &\times \prod_{j\in\{5, \cdots, L\}} N_j^{-1/2}\|\nabla Iu\|_{\infty, 2} \\
&\lesssim N^{-2+} (1+\|u_0\|_{L^2_x}).
\end{align}

If $L=4$, we have that by Hölder and Bernstein,

\begin{align}
    &\sum_{N_1\sim N_2\gtrsim N} \sum_{1<N_L\le \cdots \le N_3\ll N} \frac{1}{N_2^2} N_1 \|P_{N_1}\nabla Iu P_{N_3}  \nabla Iu\|_{2,2}\|P_{N_2}\nabla Iu P_{N_4}  Iu\|_{2,2} \prod_{j\in\{5, \cdots, 8\}}\|P_{N_j\le 1} Iu\|_{\infty, \infty}  \\
    &\lesssim \sum_{N_1\sim N_2\gtrsim N} \sum_{1<N_L\le \cdots \le N_3\ll N}\|\nabla Iu_0\|_{L^2_x}^3 \left(\|u_0\|_{L^2}+ \|P_{\le 1} I(u|u|^6)\|_{1,2}\right) \frac{N_1^{1/2}}{N_2^{5/2}}\prod_{j\in\{5, \cdots, 8\}} \|P_{N_j\le 1}Iu\|_{\infty, 8}  \\
&\lesssim N^{-2+} (1+\|u_0\|_{L^2_x}).
\end{align}

The total contribution is obtained by summing over $L$:
\begin{align}
    \sum_{L=4}^7 N^{-2+} (1+\|u_0\|_{L^2_x}) \lesssim N^{-2+} (1+\|u_0\|_{L^2_x}).
\end{align}
\textbf{Case $K$--$II_c$}: $N_1>1$ and $N_2 \gtrsim  N_3 \gtrsim N  \gg N_4$.
We estimate the multiplier as follows:
\begin{align}
    \left| 1- \frac{m(\xi_1)}{m(\xi_2)m(\xi_3)}\right| \lesssim  \frac{m(\xi_1)}{m(\xi_2)m(\xi_3)}.
\end{align}
We also assume that $N_3\ge N_4\ge \cdots\ge N_L>1=N_{L+1} \ge \cdots \ge N_8$.
We handle this case in the same manner as \textbf{Case $K$--$I_c$}, except that for $j\in\{L+1, \cdots, 8\}$, we use the estimate
\begin{align}
    \| P_{N_j\le 1} Iu\|_{\infty, \infty} \lesssim \| P_{N_j\le 1} Iu\|_{\infty, 8} \lesssim \sup_{t\in J} E(Iu(t))^{1/8} \lesssim 1.
\end{align}
Additionally if $P_{N_j\le 1}Iu$ appears in the bilinear estimate, we gain a factor of $1+\|u_0\|_{L^2_x}$.
Summing over $L\in\{3, \cdots, 7\}$, the contribution is
\begin{align}
    \lesssim N^{-2+} (1+\|u_0\|_{L^2_x}).
\end{align}
\textbf{Case $K$--$II_d$}: $N_1>1$ and $N_2 \ge  N_3 \ge N_4\gtrsim N$.

We  assume that $N_4\ge N_5\ge \cdots\ge N_L>1=N_{L+1} \ge \cdots \ge N_8$ \newline 
\textbf{Case $K$--$II_d(\alpha)$}: $N_1\sim N_2$. 
The contribution is, by Hölder and Bernstein,
\begin{align}
   &\lesssim  \sum_{N_1\sim N_2\gtrsim N}\frac{N_1}{N_2} \norm{   (P_{N_1}\, N_1 I u)}_{8,4} \norm{(P_{N_2} N_2 I u)}_{8,4} \sum_{N_3\ge N_4\gtrsim N} \sum_{N_5\ge \cdots \ge  N_{L} >1} \|P_{N_3} \nabla I u\|_{8,4}\|P_{N_4} \nabla I u\|_{8,4} \\
    &\times \frac{1}{N_3m(N_3)}\frac{1}{N_4m(N_4)} \prod_{j=5}^L \frac{1}{m(N_j)}\|P_{N_j} Iu\|_{8,\infty } \times \prod_{j=L+1}^8 \|P_{N_j} Iu\|_{8,\infty }
\end{align}

Just as we estimate in Case $I$, we use Cauchy--Schwarz first in $(N_1,N_2)$ to bound
\begin{align}
    \sum_{N_1\sim N_2\gtrsim N}\frac{N_1}{N_2} \norm{   (P_{N_1}\, N_1 I u)}_{8,4} \norm{(P_{N_2} N_2 I u)}_{8,4} \lesssim \| \nabla Iu\|_{Z(J)}^2 \lesssim  1.
\end{align}
Via Bernstein
\begin{align}
   \frac{1}{m(N_j)}\|P_{N_j} Iu\|_{8,\infty } &\lesssim  \frac{1}{m(N_j)}\|P_{N_j} |\nabla|^{1/4} Iu\|_{8,4 } \lesssim \frac{1}{N_j^{3/4}m(N_j)}, 
\end{align}

which is summable over $N_j>1$ if $s>1/4$. Summing over $L\in\{4, 5, 6, 7\}$ and $N_3, N_4$ leads to a bound of the form
\begin{align}
    \lesssim N^{-2}.
\end{align}
\textbf{Case $K$--$II_d(\beta)$}: $N_2 \sim N_3$ and $N_1\ll N_2$. 

The contribution is, for each fixed $L$, by Hölder and Bernstein,
\begin{align}
   &\lesssim  \sum_{N_2\sim N_3\gtrsim N} \sum_{N_1 \ll N_2} \frac{N_1 m(N_1)}{N_2 m(N_2)} \sum_{N_3\ge N_4\gtrsim N} \prod_{j=1}^4  \|P_{N_j} \nabla I u\|_{8,4} \\
    &\times \frac{1}{N_3m(N_3)}\frac{1}{N_4m(N_4)} \sum_{N_5\ge \cdots \ge  N_{L} >1}\prod_{j=5}^L \frac{1}{m(N_j)}\|P_{N_j} Iu\|_{8,\infty } \times \prod_{j=L+1}^8 \|P_{N_j} Iu\|_{8,\infty } \\
    &\lesssim \varepsilon^{8-L} \sum_{N_3\gtrsim N} \sum_{N_2\sim N_3} \frac{N_2^{0+}}{N_3m(N_3)} \sum_{N_4\gtrsim N} \frac{1}{N_4 m(N_4)} \\
    &\lesssim \varepsilon^{8-L} N^{-2+}.
\end{align}
Summing over $L\in\{4, 5, 6, 7\}$ gives the bound
\begin{align}
    \lesssim N^{-2+}.
\end{align}
This finishes the estimate for $K$. We now turn to estimate $P$. 
Here, we use a high-low decomposition:
\begin{align}
    u_l&= P_{\le N/30} u \\
    u_h&= P_{>N/30} u
\end{align}
and define
\begin{align}
    G(t,x) = (|u|^6u) - \mathcal{O}(u_h^6 u_l) - \mathcal{O}(u_h^7),
\end{align}
to be the terms in
\begin{align}
    (|u_h+u_l|^6(u_h+u_l))
\end{align}
that have at least two factors of $u_l$. We can estimate by the fractional Leibniz rule.
\begin{align}
    \| \nabla I G\|_{L^4_t L^2_x} &\lesssim \|\nabla Iu\|_{\infty, 2} \|Iu_l\|_{\infty, \infty}^2 \|u \|_{16, \infty}^4 + \|\nabla Iu_{l}\|_{\infty, 2} \|Iu_{l}\|_{\infty, \infty } \|u\|_{20,\infty}^5
\end{align}
Note that
\begin{align}
    \|\nabla Iu\|_{\infty, 2} \|Iu_l\|_{\infty, \infty}^2 \|u \|_{16, \infty}^4 
    &\lesssim \| \nabla I u\|_{\infty,2}^{1+2/5} \|Iu\|_{\infty, 8}^{8/5}  \left( \| Iu_{\le 1} \|_{16,\infty} + \| u_{> 1} \|_{16,\infty}\right)^4 \\
    &\lesssim \left(\| Iu_{\le 1} \|_{8,\infty}^{1/2} \| Iu_{\le 1} \|_{\infty,\infty}^{1/2} +  \sum_{N_k>1} \| P_{N_k} u \|_{16,\infty}\right)^4,
\end{align}
We now estimate by Bernstein the high frequencies as follows for $s>3/8$, 
\begin{align}
    \sum_{N_k>N} \| P_{N_k} u \|_{16,\infty} \lesssim \sum_{N_k>N} \frac{N_k^{1-s}}{N^{1-s}} \| |\nabla|^{3/8} P_{N_k} Iu \|_{16,8/3} \lesssim \sum_{N_k>N} \frac{N_k^{3/8-s}}{N^{1-s}} \| | \nabla| P_{N_k} Iu \|_{16,8/3} \lesssim N^{-5/8}
\end{align}
 and the medium frequencies 
\begin{align}
    \sum_{1<N_k\le N} \| P_{N_k} Iu \|_{16,\infty} \lesssim \sum_{1<N_k\le N} \| |\nabla|^{3/8} P_{N_k} Iu \|_{16,8/3} \lesssim \sum_{1<N_k\le N} N_k^{-5/8} \| \nabla  Iu \|_{16,8/3} \lesssim 1.
\end{align}
Using Bernstein, the defocusing nature of the energy and our Morawetz bound,  we estimate
\begin{align}
    \| Iu_{\le 1} \|_{8,\infty}^{1/2} \| Iu_{\le 1} \|_{\infty,\infty}^{1/2}&\lesssim \| Iu_{\le 1} \|_{8,10}^{1/2} \| Iu_{\le 1} \|_{\infty,8}^{1/2} \lesssim \varepsilon^{1/2} \sup_{t\in J} E(Iu(t))^{1/16} \lesssim 1 
\end{align}
Therefore, we conclude that
\begin{align}
    \| \nabla IG\|_{L^4_t L^2_x} &\lesssim 1.
\end{align}
Similarly,
\begin{align}
    \|\nabla Iu_{l}\|_{\infty, 2} \|Iu_{l}\|_{\infty, \infty } \|u\|_{20,\infty}^5 \lesssim 1
\end{align}
for $s>2/5$.
An almost identical analysis shows that 
\begin{align}
    \| \langle \nabla \rangle IG\|_{L^4_t L^2_x} &\lesssim 1+\| u_0\|_{L^2_x}
\end{align}

\textbf{Case $P$--$I_a$}: $N_1>1$ and $N_2\gtrsim N \gg N_3$. 
We bound the multiplier as follows:
\begin{align}
    \left| 1-\frac{m(\xi_2+\xi_3 + \cdots +\xi_8)}{m(\xi_2)} \right| \lesssim \frac{N_3}{N_2}.
\end{align}
\textbf{Case $P$--$I_a(\alpha)$}: $N_3=1=\ldots =N_8$. 

We can estimate the contribution of this case by Hölder and Bernstein,
\begin{align}
    &\lesssim \sum_{N_1\sim N_2\gtrsim N} \frac{1}{N_2^2 N_1} \| P_{N_1} \nabla I \,
 G(t,x)\|_{4,2} \| P_{N_2} \nabla Iu\|_{\infty, 2} \prod_{j=3}^8 \| P_{N_j\le 1} Iu\|_{8,\infty} \\
    &\lesssim \varepsilon^6 N^{-3}.
\end{align}

\textbf{Case $P$--$I_a(\beta)$}: $N_3>1=N_4=\ldots=N_8$. 

The estimate in this case is quite similar except for the fact that we lose a fractional power of $N_3$:
\begin{align}
    I_a(\beta) &\lesssim \sum_{N_1\sim N_2\gtrsim N} \sum_{N\gg N_3>1} \frac{N_3}{N_2^2 N_1} \| P_{N_1} \nabla I \,
 G(t,x)\|_{4,2} \| P_{N_2} \nabla Iu\|_{\infty, 2} \| P_{N_3} Iu\|_{8, \infty}\prod_{j=4}^8 \| P_{N_j\le 1} Iu\|_{8,\infty} \\
 &\lesssim \varepsilon^5 \sum_{N_1\sim N_2\gtrsim N} \sum_{N\gg N_3>1}\frac{1}{N_2^2 N_1} N_3^{1/4} \| P_{N_3} \nabla Iu\|_{8, 4}
 \\
    &\lesssim \varepsilon^5 N^{-11/4} \\
    &\lesssim N^{-2}
\end{align}

\textbf{Case $P$--$I_a(\gamma)$}: $N_3\ge \ldots \ge N_L>1=N_{L+1} =\ldots=N_8$ where $L\ge 4$. 

We  treat this case in a similar fashion except that for $N_{j}>1, j\ge 4$, we use the Sobolev embedding to gain a factor of $N_j^{-3/4}$ which we can sum:
\begin{align}
     I_a(\gamma) &\lesssim \sum_{N_1\sim N_2\gtrsim N} \sum_{N\gg N_3\ge \ldots \ge N_L>1} \frac{N_3}{N_2^2 N_1} \| P_{N_1} \nabla I \,
 G(t,x)\|_{4,2} \| P_{N_2} \nabla Iu\|_{\infty, 2} \| P_{N_3} Iu\|_{8, \infty}\prod_{j=4}^L \| P_{N_j} Iu\|_{8,\infty} \\
 &\times \prod_{j=L+1}^8 \| P_{N_j\le 1} Iu\|_{8,\infty} \\
 &\lesssim \sum_{N_1\sim N_2\gtrsim N} \sum_{N\gg N_3 \ge \ldots \ge N_L>1} \frac{1}{N_2^2 N_1} N_3^{1/4} \| P_{N_3} \nabla Iu\|_{8,4}\prod_{j=4}^L N_j^{-3/4}\| P_{N_j} \nabla Iu\|_{8,4}\prod_{j=L+1}^8 \| P_{N_j\le 1} Iu\|_{8,10} \\
 &\lesssim N^{-11/4} \varepsilon^{8-L}.
\end{align}
Summing over $4\le L\le 8$, we obtain a bound
\begin{align}
    I_a(\gamma) \lesssim N^{-2}.
\end{align}
\textbf{Case $P$--$I_b$}: $N_1>1$ and $N_2\ge N_3\gtrsim N$. 
We estimate the multiplier as follows:

\begin{align}
    \left| 1-\frac{m(\xi_1)}{m(\xi_2)\cdots m(\xi_8)} \right| \lesssim \frac{m(\xi_1)}{m(\xi_2)\cdots m(\xi_8)}.
\end{align}
\textbf{Case $P$--$I_b(\alpha)$}: $N_4=1=\ldots =N_8$ and $N_1\sim N_2$.
In this case,  we obtain, via Hölder and Bernstein, 
\begin{align}
    I_b(\alpha)&\lesssim \sum_{N_1\sim N_2\gtrsim N} \sum_{N_3\gtrsim N} \frac{1}{N_2  N_1} \| P_{N_1} \nabla I \,
 G(t,x)\|_{4,2} \| P_{N_2} \nabla Iu\|_{\infty, 2} \frac{1}{m(N_3)}\|  P_{N_3} Iu\|_{8, \infty}\prod_{j=4}^8 \| P_{N_j\le 1} Iu\|_{8,\infty} \\
 &\lesssim \varepsilon^{5}N^{-11/4}.
\end{align}

for $s>1/4$. 

\textbf{Case $P$--$I_b(\beta)$}: $N_4=1=\ldots=N_8$ and $N_1\ll N_2$.  By Hölder and Bernstein,
\begin{align}
I_b(\beta)\lesssim \sum_{N_1 \ll N_2} 
\sum_{\substack{N_2 \sim N_3 \geq N}} 
&\left( \frac{1}{N_1 N_2} \cdot \frac{1}{m(N_2)} \right) 
\left\| P_{N_1} \nabla I G \right\|_{4,4} 
\left\| \nabla I u_{N_2} \right\|_{\infty,2} 
\left\| I u_{N_3} \right\|_{8,4} 
\cdot \frac{1}{m(N_3)} 
\prod_{j=4}^8 \left\| I u_{N_j \leq 1} \right\|_{8,\infty} \notag \\
&\lesssim \sum_{N_3 \sim N_2\geq N} \frac{1}{N_2^s N^{1-s}} \frac{1}{N_3^{s}N^{1-s}} \notag \\
&\lesssim \sum_{N_2 \gtrsim N} \frac{1}{N_2^s N^{1-s}} \frac{1}{N_2^{s}N^{1-s}}\\
&\lesssim N^{-2} 
\end{align}
for $s>0$. 
Here, we used the fact that
\begin{align}
    N_1^{-1}\left\| P_{N_1} \nabla I G \right\|_{4,4} \lesssim N_1^{-3/4}\left\| P_{N_1} \nabla I G \right\|_{4,2} \lesssim N_1^{-3/4}.
\end{align}

\textbf{Case $P$--$I_b(\gamma)$}: $N_4\ge \ldots\ge  N_L>1=N_{L+1}=\ldots=N_8$ and $N_1\sim N_2$. We again use a similar approach and gain an additional factor of $N^{-1}$ when summing $1/N_1$. By Hölder and Bernstein, 

\begin{align}
I_b(\gamma)&\lesssim \sum_{N_4\ge \ldots \ge N_L>1} \sum_{N_1 \sim  N_2\ge N_3 \gtrsim  N} \frac{1}{N_1 N_2} 
\left\| P_{N_1} \nabla I G \right\|_{4, 2}
\left\| \nabla I u_{N_2} \right\|_{\infty, 2}
\left\| I u_{N_3} \right\|_{8, \infty}\frac{1}{m(N_3)} \frac{1}{m(N_2)} \\
&\quad \times \prod_{j=4}^{L} \frac{ \left\| I u_{N_j} \right\|_{8, \infty} }{ m(N_j) }\times
\prod_{j=L+1}^{8} \left\| I u_{N_j \leq 1} \right\|_{8, \infty} \\
&\lesssim N^{-2} N^{-3/4} \varepsilon^{8-L} \\
&\lesssim N^{-11/4}
\end{align}

for $s>1/4$ and uniformly over $4\le L\le 8$. Here, we used the fact that
\begin{align}
    \frac{\left\| I u_{N_j} \right\|_{8, \infty} }{ m(N_j) } &\lesssim \frac{1}{m(N_j)}\left\| |\nabla|^{1/4}I u_{N_j} \right\|_{8, 4} \\
    &\lesssim \frac{1}{N_j^{3/4}m(N_j)}\left\|\nabla  I u_{N_j} \right\|_{8, 4}
\end{align}
which is summable over $N_j>1$ if $s>1/4$.

\textbf{Case $P$--$I_b(\delta)$}: $N_4\ge \ldots\ge  N_L>1=N_{L+1}=\ldots=N_8$ and $N_1\ll  N_2, N_2\sim N_3$. 
We use a similar argument as was used for the previous case. By Hölder and Bernstein,
\begin{align}
I_b(\delta)&\lesssim \sum_{N_4\ge \ldots \ge N_L>1} \sum_{N_1\ll N_2} \sum_{N_2 \sim  N_3\gtrsim  N} \frac{1}{N_1 N_2} 
\left\| P_{N_1} \nabla I G \right\|_{4, 4}
\left\| \nabla I u_{N_2} \right\|_{\infty, 2}
\left\| I u_{N_3} \right\|_{8, 4}\frac{1}{m(N_3)} \frac{1}{m(N_2)} \\
&\quad \times \prod_{j=4}^{L} \frac{ \left\| I u_{N_j} \right\|_{8, \infty} }{ m(N_j) }\times
\prod_{j=L+1}^{8} \left\| I u_{N_j \leq 1} \right\|_{8, \infty} \\
&\lesssim \sum_{N_3 \sim N_2\geq N} \frac{1}{N_2^s N^{1-s}} \frac{1}{N_3^{s}N^{1-s}} \varepsilon^{8-L} \\
&\lesssim N^{-2} 
\end{align}

uniformly in $L$ for $s>1/4$.

\textbf{Case $P$--$I_c$}: $N_1=1, N_2\ge \cdots \ge N_L>1=N_{L+1} = \dots =N_8$. Note that this implies $N_2\sim N_3$ and $L\ge 3$. By Hölder and Bernstein,
\begin{align}
    I_c&\lesssim \sum_{N_4\ge \ldots \ge N_L>1} \sum_{N_2 \sim  N_3\gtrsim  N} \frac{1}{ N_2m(N_2)} 
\left\| P_{N_1\le 1}  I G \right\|_{4, 4}
\left\| \nabla I u_{N_2} \right\|_{\infty, 2}
\left\| \nabla  I u_{N_3} \right\|_{8, 4}\frac{1}{N_3 m(N_3)} \\
&\quad \times \prod_{j=4}^{L} \frac{ \left\| I u_{N_j} \right\|_{8, \infty} }{ m(N_j) }\times
\prod_{j=L+1}^{8} \left\| I u_{N_j \leq 1} \right\|_{8, \infty}
\end{align}
Here, we use Bernstein and the fact that $\langle \nabla \rangle^{-1}$ is a bounded multiplier
\begin{align}
    \left\| P_{N_1\le 1} IG \right\|_{4, 4} &\lesssim \left\| P_{N_1\le 1} IG \right\|_{4, 2} \\
    &\lesssim \left\| \langle \nabla \rangle   IG \right\|_{4, 2} \\
    &\lesssim 1+ \|u_0\|_{L^2_x}.
\end{align}

Therefore, we get a bound of 
\begin{align}
    I_c \lesssim N^{-2} (1+\|u_0\|_{L^2_x})
\end{align}
We are now left to estimate
\begin{align}
    \operatorname{Re}i \int_J \int \overline{I(u_h |u_h|^6)}(I(u|u|^6)- Iu|Iu|^6)\,dx\,dt 
    +\operatorname{Re} i\int_J \int \overline{I(3|u_h|^4 u_h^2 \overline{u_l} + 4 u_l |u_h|^6)}(I(u|u|^6)- Iu|Iu|^6)\,dx\,dt.
\end{align}
We evaluate 
\begin{align}
      \operatorname{Re} i\int_J \int \overline{I(u_h |u_h|^6)}I(u|u|^6)\,dx\,dt 
    +\operatorname{Re} i\int_J \int \overline{I(3|u_h|^4 u_h^2 \overline{u_l} + 4 u_l |u_h|^6)}I(u|u|^6)\,dx\,dt
\end{align}
by taking advantage of cancellations. We may write
\begin{align}
    I(u|u|^6)=  I(u_h |u_h|^6) + I(3|u_h|^4 u_h^2 \overline{u_l} + 4 u_l |u_h|^6)+ IG(t,x)
\end{align}
and notice that 
\begin{align}
      &\operatorname{Re} i\int_J \int \overline{I(u_h |u_h|^6)}\left[I(u_h |u_h|^6) + I(3|u_h|^4 u_h^2 \overline{u_l}) + I(4 u_l |u_h|^6)\right]\,dx\,dt 
    \\
    &+\operatorname{Re} i\int_J \int \overline{I(3|u_h|^4 u_h^2 \overline{u_l} + 4 u_l |u_h|^6)}\left[I(u_h |u_h|^6) + I(3|u_h|^4 u_h^2 \overline{u_l}) + I(4 u_l |u_h|^6)\right]\,dx\,dt=0
\end{align}
due to the fact that for any complex numbers $a, b$, we have
\begin{align}
    \operatorname{Re} ia\overline{a}&=0\\
    \operatorname{Re} i(a\overline{b}+\overline{a}b)&=0.
\end{align}
As a result, we are left to evaluate
\begin{align}
    \operatorname{Re} \int_J \int i \overline{I(u_h |u_h|^6)} IG(t,x)\,dx\,dt + \operatorname{Re} \int_J \int i \overline{I(O(u_h^6 u_l))} IG(t,x)\,dx\,dt
\end{align}
At this point, we recall  the following bound valid for $s\ge \sigma$ and $1< p< \infty$ that is stated in Lemma \Ref{I-operator lemma}: 
\begin{align}
    \| |\nabla|^\sigma u_{>N/30}\|_p \lesssim N^{\sigma-1} \| \nabla Iu\|_p.
\end{align}

Note that by Hölder and Bernstein, 
\begin{align}
    \left| \operatorname{Re} \int_J \int i \overline{I(u_h |u_h|^6)} IG(t,x)\,dx\,dt\right| &\le \|IG(t,x)\|_{4,2} \| u_h\|_{4+, \infty-}^3 \|u_h\|^4_{\infty- ,8+} \\
    &\lesssim \,(1+\| u_0\|_{L^2_x}) N^{-3} \| |\nabla|^{3/8} u_h\|^4_{\infty- ,2+} \\
    &\lesssim (1+\| u_0\|_{L^2_x}) N^{-3} (N^{-5/8})^4 \\
     &\lesssim (1+\| u_0\|_{L^2_x}) N^{-5} 
    \end{align}
for $s>3/8$. Similarly, by H\"older and Bernstein,
\begin{align}
     \left| \operatorname{Re} \int_J \int i \overline{I(O(u_h^6 u_l))} IG(t,x)\,dx\,dt\right| &\le \|IG(t,x)\|_{4,2} \| u_h\|_{8,12}^6 \|Iu_l\|_{\infty ,\infty} \\
     &\lesssim (1+\|u_0\|_{L^2_x}) \||\nabla|^{1/6} u_h\|_{8,4}^6  \\
     &\lesssim  (1+\|u_0\|_{L^2_x}) (N^{-5/6})^6 \\
     &\le(1+\|u_0\|_{L^2_x}) N^{-5}
\end{align}
for $s>1/6$. 
Finally, we bound by Hölder and Bernstein,
\begin{align}
    &\left|\operatorname{Re} \int_J \int i I(u_h |u_h|^6) Iu|Iu|^6 + iI(O(u_h^6 u_l)) Iu |Iu|^6\,dx\,dt\right| \\
    &\le \| u_h\|_{7, 7}^7 \| Iu\|_{\infty, \infty}^7 + 
\| u_h\|_{6, 6}^6 \|Iu_l\|_{\infty, \infty} \| Iu\|_{\infty, \infty}^7 \\
&\lesssim \| |\nabla|^{1/14} u_h\|_{7, 14/3}^7  + N^{-6}\\
&\lesssim (N^{-13/14})^7 + N^{-6} \\
&\lesssim N^{-6}
\end{align}
for $s>1/14$. 
\end{proof}

\section{Almost Morawetz Estimates} \label{Almost Morawetz Estimate}
We recall the $L^{p+3}_{t,x}$ Morawetz estimate for solutions to the one-dimensional defocusing $p$-NLS proved in \cite{PV09}:
\begin{align}
\int_{[0,T] \times \R}  \abs{u}^{p+3}\,dt\,dx \lesssim_p \norm{u_0}_{L^2(\R)}^2 \norm{u(t)}_{L^\infty_t \dot{H}_x^{\frac{1}{2}}([0,T]\times \R)}^2
\end{align}
 We now prove our low-regularity almost Morawetz estimate.
\begin{proposition}\label{lem Morawetz}
Let $u$ be a solution to defocusing $p$-NLS \eqref{p-NLS}. Then, with
\begin{align}
    F(u)= Iu |Iu|^{p-1} - I(u|u|^{p-1}),
\end{align}
the following inequality holds:
    \begin{align}
        \norm{Iu}_{L^{p+3}_{t,x}([0,T]\times\R)}^{p+3} &\lesssim_p \norm{u_{0}}_{ L^2_x( \R)}^3 \sup_{t\in[0,T]} E(Iu)(t)^{1/2} +\norm{u_{0}}_{ L^2_x(\R)}^2 \sup_{t\in[0,T]} E(Iu)(t)^{1/2} \norm{F(u)}_{L^1_t L^2_x ([0,T]\times \R)}
    \end{align}
\end{proposition}
\begin{remark}
Note that the second term on the right-hand side involving the $L^1_t L^2_x ([0,T]\times \R)$ norm of $F(u)$ should be considered as an error term and thus is expected to be perturbative. 
\end{remark}
\begin{proof}[Proof of Proposition \Ref{lem Morawetz}]
We modify the approach of \cite{PV09} to deal with the $I$ multiplier. We define, with $a(x)=|x|$,
    \begin{align*}
        M_t(Iu) := 2\int \partial_x a(x-y) |Iu|^2(y) \im\left(\overline{Iu} \,\partial_x (Iu)\right)\,dx\,dy.
    \end{align*}
    Note that by Cauchy--Schwarz,
    \begin{align}
        |M_t(Iu)|\lesssim \norm{u_0}_{L^2(\R)}^3 \norm{Iu}_{L^\infty_t \dot{H}^1_x([0,T]\times \R)}
    \end{align}
    Additionally, we will repeatedly use
    \begin{align}
        \partial_t Iu = i\Delta Iu - i Iu |Iu|^{p-1} +i F(u),
    \end{align}
    where we think of $F(u)$ as an error:
    \begin{align}
        F(u) = Iu |Iu|^{p-1} - I(u|u|^{p-1}).
    \end{align}
    If $F(u)=0$, then the calculations would proceed exactly as in \cite{PV09}  and we would obtain their Morawetz estimate replacing $u$ by $Iu$ everywhere. Since this is not the case, we carefully repeat their calculation while keeping track of the additional error terms due to $F(u)$. To organize the calculations, we break down the derivative of $M_t(Iu)$ as a sum of three terms:
    \begin{align}
        \frac{d}{dt} M_t(Iu) = K + A + E 
    \end{align}
    where $K$ represents the contributions from the linear part of the equation, $A$ the contributions from the nonlinear part of the equation and $E$ the error terms. Now, we further break down $K, A, E$ in terms of several terms for each.
    \begin{align}
        K&= K_1 + K_2 + K_3, \\
        A&=A_1 + A_2 \\
        E&= E_1+ 2E_2 + E_3. 
    \end{align}
    The first case to consider is when $\partial_t$ comes inside the integral for $M_t$ and acts on $|Iu|^2(y)$:
    \begin{align}
          2 \int_{\R\times \R} (\partial_x a) \im\left(\overline{Iu} \,\partial_x (Iu)\right) \partial_t |Iu|^2(y)\,dx\,dy&= K_1+E_1,
    \end{align}
    where we have presciently written the above contribution as a linear term $K_1$ and an error term $E_1$.
    Now, 
    \begin{align}
        \partial_t |Iu|^2 &= 2 \Re\left[ \overline{Iu} \,\partial_t Iu\right] \\
        &= - 2\im \left[\overline{Iu} \Delta Iu - |Iu|^{p+1} + \overline{Iu} F(u)\right] \\
        &= -2 \partial_y \im (\overline{Iu}\partial_y Iu) - 2 \im \left[\overline{Iu} F(u)\right]. 
    \end{align}
    Hence, we have our first error term:
    \begin{align}
        E_1 = -4 \int_{\R\times \R} (\partial_x a) \im\left(\overline{Iu} \,\partial_x (Iu)\right)\im \left[\overline{Iu} \,F(u)\right](y) \,dx\,dy.
    \end{align}
    The other term is
    \begin{align}
        K_1=-4  \int_{\R\times \R} (\partial_x a) \im\left(\overline{Iu} \,\partial_x (Iu)\right) \partial_y \im (\overline{Iu}\partial_y Iu)\,dx\,dy 
        &= -4 \int_{\R\times \R} (\partial_x^2 a) \im\left(\overline{Iu} \,\partial_x (Iu)\right)  \im (\overline{Iu}\partial_y Iu)\,dx\,dy \\
        &= -8 \int_{ \R} \left(\im\overline{Iu} \,\partial_x (Iu)\right)^2 \,dx,
    \end{align}
    where we used 
    \begin{align}
    \partial_{xx} a(x-y)=2\delta_y.    
    \end{align}

Now, the second possibility is when $\partial_t$ acts on the $\im\left[ \overline{Iu} \,\partial_x (Iu)\right]$ factor inside the integral:
\begin{align}
  &\quad 2\int_{\R\times \R} (\partial_x a) \partial_t\im\left[ \overline{Iu} \,\partial_x (Iu)\right] |Iu|^2(y)\,dx\,dy
    \\
    &= -i\int_{\R\times \R} (\partial_x a) \partial_t \left[ \overline{Iu} \,\partial_x (Iu) - Iu \partial_x \overline{Iu}\right]|Iu|^2(y)\,dx\,dy \\
    &= -i\int_{\R\times \R} (\partial_x a)  \left[ \overline{\partial_t Iu} \,\partial_x (Iu) + \overline{Iu} \partial_t \partial_x Iu  - \partial_t Iu  \overline{\partial_x Iu} - Iu \overline{\partial_t \partial_x Iu}  \right]|Iu|^2(y)\,dx\,dy.
\end{align}
We first focus on the first and third term, we again presciently define  
\begin{align}
    K_2 + A_1 +E_2 &:= \int_{\R\times \R} |Iu|^2(y)\partial_x a \left[i\partial_t Iu  \overline{\partial_x Iu} -i \overline{\partial_t Iu} \,\partial_x (Iu)\right] \,dx\,dy \\
    &= \int_{\R\times \R} |Iu|^2(y)\partial_x a \left[i\partial_t Iu  \overline{\partial_x Iu} +\overline{i\partial_t Iu} \,\partial_x (Iu)\right] \,dx\,dy \\
    &= \int_{\R\times \R} |Iu|^2(y)\partial_x a \left(\left[ -\partial_x^2 Iu + Iu |Iu|^{p-1} - F(u)\right] \partial_x \overline{ Iu} + \left[ -\partial_x^2 \overline{Iu} + \overline{Iu} |Iu|^{p-1} - \overline{F(u)}\right]  \partial_x Iu\right) \,dx\,dy. \\
\end{align}
We focus on 
\begin{align}
    K_2 &= -\int_{\R\times \R} |Iu|^2(y)\partial_x a \left( \partial_x^2 Iu \partial_x \overline{Iu} + \partial_x^2 \overline{Iu} \partial_x Iu\right)\\
    &=-\int_{\R\times \R} |Iu|^2(y)\,\partial_x a\, \partial_x |\partial_x Iu|^2 \,dx\,dy \\
    &= \int_{\R\times \R} \partial_{x}^2 a  |Iu|^2(y) |\partial_x Iu|^2 \,dx\,dy \\
    &= 2 \int_{\R} |Iu \,\partial_x Iu|^2 \,dx.
\end{align}
The nonlinear term is
\begin{align}
    A_1 &= \int_{\R\times \R}  |Iu|^2(y)\,\partial_x a |Iu|^{p-1}(x) \partial_x |Iu|^2\,dx\,dy \\
    &= \frac{2}{p+1} \int_{\R\times \R} |Iu|^2(y)\,\partial_x a \,\partial_x |Iu|^{p+1}\,dy\,dx \\
    &= -\frac{2}{p+1} \int_{\R\times \R} \partial_{xx} a |Iu|^2(y)|Iu|^{p+1}(x)\,dx\,dy \\
    &= -\frac{4}{p+1} \int_{\R} |Iu|^{p+3}\,dx.
\end{align}
Thirdly, we have our second error term
\begin{align}
     E_2 =-2\int_{\R\times \R}\partial_x a |Iu|^2(y) \Re\left[F(u) \partial_x \overline{Iu}\right]\,dx\,dy.
\end{align}
The second and fourth terms are similarly defined as 
\begin{align}
    K_3 +A_2 +E_3  &= \int_{\R\times \R} (\partial_x a)  \left[    i Iu \partial_t \partial_x \overline{Iu}  -i \overline{Iu} \partial_t \partial_x Iu \right]|Iu|^2(y)\,dx\,dy \\
    &= -\int_{\R\times \R} i\partial_t \overline{Iu} \,\partial_x \left(\partial_x a Iu\right) |Iu|^2(y)\,dx\,dy + \int_{\R\times \R} i \partial_t Iu \partial_x \left( \partial_x a\, \overline{Iu}\right)|Iu|^2(y)\,dx\,dy \\
    &=-\int_{\R\times \R} i\partial_t \overline{Iu} \left(2\delta_{x=y} Iu + \partial_x a \,\partial_x Iu \right) |Iu|^2(y)\,dx\,dy + \int_{\R\times \R} i \partial_t Iu \left( 2\delta_{x=y} \, \overline{Iu} + \partial_x a \,\partial_x \overline{Iu} \right)|Iu|^2(y)\,dx\,dy \\
    &= K_2 + A_1 + E_2 + 2\int_{\R} \left(i\partial_t Iu \overline{Iu} + \overline{i\partial_t Iu} Iu \right)|Iu|^2\,dx \\
    &= K_2 + A_1 + E_2 + 4\int_{\R} \Re \left(i\partial_t Iu \,\overline{Iu}\right)|Iu|^2\,dx \\
    &= K_2 + A_1 + E_2 + 4\Re\int_{\R}  \left( -\partial_{xx} Iu + Iu |Iu|^{p-1} - F(u)\right)\overline{Iu}|Iu|^2\,dx \\
    &= K_2 + A_1 + E_2 + 4\int_{\R} |Iu|^{p+3}\,dx + E_3 + 4 \Re \int_{\R} \partial_x Iu \,\partial_x \left(\overline{Iu} |Iu|^2 \right)\,dx \\
    &= K_2 + A_1 + E_2 + 4\int_{\R} |Iu|^{p+3}\,dx + E_3 + 4 \int_{\R} |Iu\,\partial_x Iu|^2\,dx + 2 \int_{\R} \left(\partial_x Iu \overline{Iu} + \partial_x \overline{Iu} Iu\right) \partial_x |Iu|^2 \,dx \\
    &= K_2 + A_1 + E_2 + 4\int_{\R} |Iu|^{p+3}\,dx + E_3 + 4 \int_{\R} |Iu\,\partial_x Iu|^2\,dx + 2 \int_{\R} (\partial_x |Iu|^2)^2 \,dx \\
    &= 3K_2 + A_1 + E_2 + 4\int_{\R} |Iu|^{p+3}\,dx + E_3  + 2 \int_{\R} (\partial_x |Iu|^2)^2 \,dx,
\end{align}
where our third and last error term is
\begin{align}
    E_3 = -4 \int_{\R} |Iu|^2 \Re\left(\overline{Iu} \,F(u)\right)\,dx.
\end{align}
Thus,
\begin{align}
   K+A &= 8 \int_{\R} \left(\Re \overline{Iu} \,\partial_x Iu\right)^2 \,dx + 2\int_{\R} (\partial_x |Iu|^2)^2 \,dx + 4\left(1-\frac{2}{p+1}\right)\int_{\R} |Iu|^{p+3}\,dx.
\end{align}
Thus, integrating $\frac{d}{dt} M_t(Iu)$ on $[0,T]$, we have that 
\begin{align}
    \frac{4(p-1)}{p+1}\norm{Iu}^{p+3}_{L^{p+3}_{t,x}([0,T]\times \R)} + \int_0^T E_1 + 2E_2 + E_3\,dt &\lesssim \sup_{t\in[0,T]} |M_t(Iu)| \\
    &\lesssim \norm{u_0}_{L^2_x (\R)}^3 \norm{Iu}_{L^\infty_t \dot{H}_x^1 ([0,T]\times \R)}.
\end{align}
Now, we estimate the $L^1([0,T])$ norms of the error terms using the fact that $\norm{\partial_x a(x-y)}_{L_{x,y}^\infty(\R^2)}\le 1$ and Cauchy--Schwarz. First,
\begin{align}
    \left|\int_0^T E_1 \,dt\right| &\lesssim \norm{u_0}_{L^2_x(\R)}^2 \norm{Iu}_{L^\infty_t \dot{H}_x^1 ([0,T]\times \R)} \norm{F(u)}_{L^1_t L^2_x ([0,T]\times \R)}.
\end{align}
Now, we have the same bound for $E_2$,
\begin{align}
        \left|\int_0^T E_2 \,dt\right| &\lesssim \norm{u_0}_{L^2_x(\R)}^2 \norm{Iu}_{L^\infty_t \dot{H}_x^1 ([0,T]\times \R)} \norm{F(u)}_{L^1_t L^2_x ([0,T]\times \R)}.
\end{align}
For $E_3$, we use Cauchy--Schwarz to obtain
\begin{align}
            \left|\int_0^T E_3 \,dt\right| &\lesssim \norm{Iu}_{L^\infty_t L^6_x([0,T]\times \R)}^3 \norm{F(u)}_{L^1_t L^2_x ([0,T]\times \R)} 
\end{align}
Now, by Sobolev embedding and interpolation,
\begin{align}
    \norm{Iu}_{L^\infty_t L^6_x([0,T]\times \R)} &\lesssim \norm{|\nabla|^{1/3} Iu}_{L^\infty_t L^2_x([0,T]\times \R)} \\
    &\lesssim \norm{u_0}_{L^2_x(\R)}^{2/3}\norm{Iu}_{L^\infty_t \dot{H}_x^1 ([0,T]\times \R)}^{1/3}
\end{align}

Hence, 
\begin{align}
    \left|\int_0^T E_3 \,dt\right| &\lesssim \norm{u_0}_{L^2_x(\R)}^{2}\norm{Iu}_{L^\infty_t \dot{H}_x^1 ([0,T]\times \R)} \norm{F(u)}_{L^1_t L^2_x ([0,T]\times \R)}. 
\end{align}
This finishes the proof of Proposition \Ref{lem Morawetz}.
\end{proof}

\begin{lemma}[Estimate for error term]\label{lem Error}
For $s>1/5$, an interval $J=[t_1,t_2]$ such that $\sup_{t\in J}E(Iu(t))\le 1$ and $\norm{Iu}_{L^{8}_{t}L^{10}_{x}(J\times \R)}\le \e$, for all large enough $N$ and small enough $\varepsilon$
    \begin{align}
        \norm{F(u)}_{L^1_t L^2_x (J\times \R)} \lesssim  N^{-2+}
    \end{align}
    where $F(u)$ is as defined in Proposition \Ref{lem Morawetz} in the case $p=7$. 
\end{lemma}
\begin{proof}[Proof of Lemma \Ref{lem Error}]
We use a Littlewood--Paley decomposition and proceed by cases, in each one implicitly using Proposition \ref{multilinear prop}.  Let 
\begin{align}
    u = P_1u+\sum_{N\in 2^{\Z}, N> 1} P_N u,
\end{align}
where $P_1 u=P_{\le 1}u$. We will write $u_N := P_N u$. Then, by duality and Plancherel,
\begin{align}
\norm{F(u)}_{L^1_t L^2_x (J\times \R)} &= \sup_{\|\phi\|_{L^\infty_t L^2_x}=1} \left|\int_J \langle \phi, F(u)\rangle_{L^2_x}\,dt\right| \\
&= \sup_{\|\phi\|_{L^\infty_t L^2_x}=1}  \bigg|\int_J \int_{\sum_{j=0}^7 \xi_j=0} \left[1-\frac{m(\xi_1 + \dots + \xi_7)}{m(\xi_1) \cdots m(\xi_7)}\right] \widehat{\overline{\phi}}(t,\xi_0) \widehat{Iu}(t, \xi_1) \cdots  \\
    &\times  \widehat{\overline{Iu}}(t, \xi_6)\widehat{Iu}(t, \xi_7)  d\xi_0 \dots d\xi_7\,dt\bigg|
\end{align}
Hence, we have that, by a Littlewood--Paley decomposition, 
\begin{align}
&\norm{F(u)}_{L^1_t L^2_x (J\times \R)} \le \sum_{N_1\ge \ldots \ge N_7} \sup_{\|\phi\|_{L^\infty_t L^2_x}=1} \bigg|\int_J  \int_{\sum_{j=0}^7 \xi_j=0} \left[1-\frac{m(\xi_1 + \dots + \xi_7)}{m(\xi_1) \cdots m(\xi_7)}\right] \widehat{\overline{\phi}}(t,\xi_0)\widehat{Iu_{N_1}}(t, \xi_1) \cdots  \\
    &\times  \widehat{\overline{Iu_{N_6}}}(t, \xi_6)\widehat{Iu_{N_7}}(t, \xi_7)  d\xi_0 \dots d\xi_7\,dt\bigg|
\end{align}
Note that due to Lemma \ref{lem LWP}, we have that
\begin{align}
    \|\nabla Iu\|_{Y(J)}\lesssim 1.
\end{align}
\textbf{Case} $I$ ($N_j>1$ for each $j$ ). \newline 
\textbf{Case} $I_a$ ($N_1 \ll N$) In this case, $m(\xi_i)=1$ for all $i$, so
    \begin{align}
        1-\frac{m(\xi_1+ \cdots + \xi_7)}{m(\xi_1)\cdots m(\xi_7)}=0.
    \end{align}
\textbf{Case} $I_b$ ($N_1 \gtrsim N \gg N_2$). We have that in this region,
\begin{align}
             \left\|1-\frac{m(\xi_1+ \cdots + \xi_7)}{m(\xi_1)\cdots m(\xi_7)}\right\|_{L^\infty} \lesssim \frac{N_2}{N_1},
\end{align}
along with all its higher scaled derivatives, as in the case $I_b$ in the proof of Theorem \ref{Energy Increment Theorem}. Additionally, since $\xi_0=-(\xi_1+\dots +\xi_7)$ and $N_2\ll N_1$, we have that $\phi(\xi_0)$ is automatically localized to $|\xi_0|\sim N_1$ so we may insert a slightly enlarged projection $\widetilde{P}_{N_1}\phi$ for $\phi$ in our integral. Hence, applying Proposition \ref{multilinear prop}, we have that the contribution of Case $I_b$ is
    \begin{align}
     &\lesssim \sup_{\|\phi\|_{L^\infty_t L^2_x(J\times \R)}=1}\sum_{N_1 \gtrsim N} \sum_{N_j \ll N, j\ge 2} \frac{N_2}{N_1}  \frac{1}{ N_1  N_2  N_3  N_4  N_5 } \|\phi\|_{L^\infty_t L^2_x(J\times \R)} \prod_{l=1}^4 \norm{P_{N_l} \nabla Iu}_{L^4_t L^\infty_x(J\times \R)} \\
     &\times \norm{P_{N_5} \nabla Iu}_{L^\infty_t L^2_x(J\times \R)} \prod_{m=6}^7 \norm{P_{N_m}  Iu}_{L^\infty_t L^\infty_x(J\times \R)} \\
    &\lesssim  \sum_{N_1 \gtrsim N}\frac{1}{N_1^2} \sum_{N_j \ll N, j\ge 2} \frac{1}{  N_3  N_4  N_5 N_6^{1/2}  N_7^{1/2} } \\
     &\lesssim N^{-2+}.
        \end{align}
    \textbf{Case } $I_c$ ($N_1 \ge N_2\gtrsim N$).    
    In this case, there is no cancellation to exploit so we separately consider the multipliers $\sigma^1_N(\xi):= 1$ and $\sigma^2_N(\xi):= \frac{m(\xi_0)}{m(\xi_1)\cdots m(\xi_7)}$. The first satisfies Condition \eqref{symbol condition} with $C_{\mathbf{N}}=1$ while for the second, we use $m(\xi_0) \widehat{\overline{\phi}}(\xi_0)= \widehat{\overline{I\phi}}(\xi_0)$ and go back to physical space. Letting $K_j= \widecheck{\left(m(\cdot)^{-1} \psi_{N_j}(\cdot)\right)}$, the remaining part to estimate is 
    \begin{align}\label{I_3 lemma 5.3}
     \sup_{\|\phi\|_{L^\infty_t L^2_x(J\times \R)}=1} \left|\int_{J\times \R} \overline{I\phi}  (K_1 \ast I u_{N_1}) \dots (K_6 \ast Iu_{N_6}) (K_7 \ast Iu_{N_7})\,dx\,dt\right|.
    \end{align}
    We note that $K_j$ satisfies
    \begin{align} 
        \|K_j\|_{L^1_x}\lesssim \frac{1}{m(N_j)}
    \end{align}
    due to properties of the multiplier $m(\cdot)$ and therefore, by the Holder, Young and Bernstein inequalities,  
    \begin{align}
        \eqref{I_3 lemma 5.3} &\lesssim  \sup_{\|\phi\|_{L^\infty_t L^2_x(J\times \R)}=1} \|I\phi\|_{L^\infty_t L^2_x(J\times \R)} \frac{1}{N_1 m(N_1) N_2 m(N_2)} \norm{P_{N_1} \nabla Iu}_{L^6_t L^6_x(J\times \R)}\norm{P_{N_2} \nabla Iu}_{L^6_t L^6_x(J\times \R)} \\
        &\times \prod_{j=3}^7 \frac{1}{m(N_j)}\norm{P_{N_j} Iu}_{L^{15/2}_t L^{30} _x(J\times \R)}.
    \end{align}
    A similar bound holds for the first term but without the factors of $m(N_j)^{-1}$ and since $m(N_j)\le 1$, that contribution will always be less than the second term. 
    Hence, the contribution from Case $I_c$
    \begin{align}
        &\lesssim \sum_{N_1\ge N_2 \gtrsim N} \frac{1}{N_1 m(N_1) N_2 m(N_2)} \norm{P_{N_1} \nabla Iu}_{L^6_t L^6_x(J\times \R)}\norm{P_{N_2} \nabla Iu}_{L^6_t L^6_x(J\times \R)} \\
        &\times \sum_{N_7\le \ldots \le N_2} \prod_{j=3}^7 \frac{1}{m(N_j)}\norm{P_{N_j} Iu}_{L^{15/2}_t L^{30} _x(J\times \R)}  \\
        &\lesssim \sum_{N_1\ge N_2 \gtrsim N} \frac{N_1^{-s}}{N^{1-s}}  \frac{N_2^{-s}}{N^{1-s}} \sum_{N_7\le \ldots \le N_2} \prod_{j=3}^7 \frac{1}{m(N_j)N_j^{4/5}}\norm{P_{N_j} \nabla Iu}_{L^{15/2}_t L^{30/7} _x(J\times \R)} \\
        &\lesssim \frac{1}{N^2}
    \end{align}
    as $s>1/5$. \newline
    \textbf{Case} $II$ (There exists $j_0$ such that $1\le j_0\le 7$ and $N_{j_0}=1$).
    \newline
    \textbf{Case} $II_a$ ($N_1\ll N$). This case contributes zero again. \newline
    \textbf{Case} $II_b$ ($N_1\gtrsim  N \gg N_2 $). \newline
     Let $L$ be such that $N_1\ge \cdots \ge N_L>1=N_{L+1} \ge \cdots \ge N_7$. We again estimate
     \begin{align}
         \left| 1-\frac{m(\xi_1+ \cdots + \xi_7)}{m(\xi_1)\cdots m(\xi_7)}\right| \lesssim \frac{N_2}{N_1}.
     \end{align}
     
The contribution from this case for each fixed $L$ is
\begin{align}
    &\sum_{N_1 \gtrsim N}\sum_{1<N_2\ll N} \frac{1}{N_1^2} \|P_{N_1} \nabla Iu\|_{L^{4}_t L^{\infty}_x(J\times \R)}\|P_{N_2} \nabla Iu\|_{L^{8}_t L^{4}_x(J\times \R)} \\
    &\times \sum_{N_L\le \ldots \le N_3\le N_2} \prod_{j=3}^L \|P_{N_j} Iu\|_{L^{8}_t L^{20} _x(J\times \R)} \cdot \prod_{j=L+1}^7 \|P_{N_j\le 1} Iu\|_{L^{8}_t L^{20} _x(J\times \R)} \\
    &\lesssim N^{-2}\log N\, \varepsilon^{7-L} \\
    &\lesssim N^{-2+}
\end{align}
uniformly in $L\ge 2$. When $L=1$, i.e.\ $N_2=1$, all six factors $j=2,\dots,7$ are low-frequency and we instead apply H\"older with $P_{N_1}\nabla Iu$ in $L^4_tL^\infty_x$ and each $P_{N_j\le 1}Iu$ in $L^8_tL^{12}_x$; since $\|P_{\le 1}Iu\|_{L^8_tL^{12}_x(J\times\R)}\lesssim \|P_{\le 1}Iu\|_{L^8_tL^{10}_x(J\times\R)}\lesssim \varepsilon$ by Bernstein, the contribution is again $\lesssim \sum_{N_1\gtrsim N} N_1^{-2}\varepsilon^{6}\lesssim N^{-2}\varepsilon^{7-L}$.\newline
\textbf{Case} $II_c$ ($N_1 \ge  N_2 \gtrsim N  $). \newline
     Let $L$ be such that $N_3\ge \cdots \ge N_L>1=N_{L+1} \ge \cdots \ge N_7$. Applying an almost identical analysis as in Case $I_c$, we have that the contribution is
\begin{align}
    &\lesssim \sum_{N_1\ge N_2 \gtrsim N} \frac{1}{N_1 m(N_1) N_2 m(N_2)} \norm{P_{N_1} \nabla Iu}_{L^{16/3}_t L^{8}_x(J\times \R)}\norm{P_{N_2} \nabla Iu}_{L^{16/3}_t L^{8}_x(J\times \R)} \\
    &\times \sum_{N_7\le \ldots \le N_2} \prod_{j=3}^L \frac{1}{m(N_j)}\norm{P_{N_j} Iu}_{L^{8}_t L^{20} _x(J\times \R)} \cdot \prod_{j=L+1}^7 \frac{1}{m(N_j)}\norm{P_{N_j\le 1} Iu}_{L^{8}_t L^{20} _x(J\times \R)} \\
    &\lesssim \sum_{N_1\ge N_2 \gtrsim N} \frac{1}{N_1 m(N_1) N_2 m(N_2)}  \\
    &\times \sum_{N_7\le \ldots \le N_2} \prod_{j=3}^L \frac{1}{N_j^{4/5}m(N_j)}\norm{P_{N_j} \nabla Iu}_{L^{8}_t L^{4} _x(J\times \R)} \cdot \prod_{j=L+1}^7 \norm{P_{N_j\le 1} Iu}_{L^{8}_t L^{10} _x(J\times \R)}  \\
    &  \lesssim N^{-2} \e^{7-L} 
\end{align}
for $s>1/5$.
Thus, the total contribution from this case is
\begin{align}
    N^{-2} \sum_{L=2}^7  \e^{7-L}  \lesssim N^{-2}
\end{align}
This finishes the proof of Lemma \ref{lem Error}. 
\end{proof}
\section{Proof of Main Theorem}
\label{Proof of Main Theorem}

In this section, we will present the proof of the main theorem. 
\subsection{Global Well-Posedness}
If $u(t,x)$ is a solution to \eqref{septic NLS} on $[0,T_0]$, then $u^\lambda(t,x) :=\lambda^{-1/3} u(t/\lambda^2, x/\lambda)$ is a solution on $[0, \lambda^2 T_0]$. Note that
\begin{align}
    \int_{\R} |\nabla I u_0^\lambda|^2\,dx  &\lesssim (N^{1-s} \|u_0^\lambda\|_{\dot{H}^s(\R)})^2 = (N^{1-s} \lambda^{1/6-s}\|u_{0}\|_{\dot{H}^s(\R)})^2.
\end{align}
and 
\begin{align}
    \| Iu_0^\lambda\|_{L^8_x}^8 \lesssim  \| Iu_0^\lambda\|_{\dot{H}^1_x}^2\| Iu_0^\lambda\|_{\dot{H}^{1/6}_x}^6\le \| Iu_0^\lambda\|_{\dot{H}^1_x}^2\| u_0^\lambda\|_{\dot{H}^{1/6}_x}^6= \| Iu_0^\lambda\|_{\dot{H}^1_x}^2 \|u_0\|_{\dot{H}^{1/6}_x}^6  \end{align}
using the fact that $1/6$ is the scaling critical exponent. 
Let 
\begin{align}
    \lambda= C(s, \|u_{0}\|_{\dot{H}^s(\R)},\|u_0\|_{\dot{H}^{1/6}_x}) N^{\frac{1-s}{s-1/6}}
\end{align}
so that
\begin{align}
    E(Iu_0^\lambda)\le \frac{1}{2}. 
\end{align}
In what follows we write $u$ for $u^\lambda$. Define
\begin{align}
    W= \left\{\tau\in [0, \lambda^2T_0]: \sup_{t\in[0,\tau ]} E(Iu^\lambda(t))\le \frac{9}{10}\right\} \subseteq [0, \lambda^2 T_0]
\end{align}

The goal is to show that $W=[0, \lambda^2 T_0]$. We do so by showing that $W$ is non-empty, and both open and closed, hence must be the whole of the connected set $[0, \lambda^2 T_0]$ which contains $W$.
Clearly, $0\in W$ and $W$ is closed by the Dominated Convergence Theorem. We now show that $W$ is open in $[0, \lambda^2 T_0]$. Assume that $W=[0,T]$ for $T<\lambda^2 T_0$, then there exists $\delta>0$ such that 
\begin{align}
    E(Iu^\lambda(t))\le 1
\end{align}
on $[0, T+\delta)$. We state the following lemma guaranteeing control of the $L^{10}_{t,x}$ Morawetz norm of $Iu$:
\begin{lemma}
    For $s>29/79, T+\delta\le \lambda^2 T_0$, we have that there exists $C>0$ such that
    \begin{align}
        \| Iu^\lambda\|_{L^{10}_{t,x}([0, T+\delta]\times \R)}^{10} \le \frac{3}{2}C\lambda^{1/2} \|u_0\|_{L^2_x}^3
    \end{align}

\end{lemma}
\begin{proof}
    Let 
    \begin{align}
    Z=\{t\in[0, T+\delta]: \| Iu^\lambda\|_{L^{10}_{t,x}([0, t]\times \R)}^{10} \le \frac{3}{2}C\lambda^{1/2} \|u_0\|_{L^2_x}^3\}
    \end{align}
    
    As usual, $Z$ is non-empty and closed. To show it is open, assume that $Z=[0, T'] $.  Then, there exists $\delta'$ such that 
    \begin{align}
        \| Iu^\lambda\|_{L^{10}_{t,x}([0, T'+ \delta']\times \R)}^{10}  \le 2C\lambda^{1/2}\|u_0\|_{L^2_x}^3.
    \end{align}
    Therefore, by the Hölder inequality, 
    \begin{align}
        \| Iu^\lambda\|_{L^8_t L^{10}_x([0, T'+\delta']\times\R)}^8 \le (\lambda^2T_0)^{1/5}\| Iu^\lambda\|_{L^{10}_t L^{10}_x([0, T'+\delta']\times\R)}^8 \le (\lambda^2T_0)^{1/5} ( 2C\lambda^{1/2}\|u_0\|_{L^2_x}^3)^{4/5}.
    \end{align}
    
    Divide $[0, T'+\delta']$ into 
    \begin{align}
       M=\left\lceil\frac{(\lambda^2T_0)^{1/5} ( 2C\lambda^{1/2}\|u_0\|_{L^2_x}^3)^{4/5}}{\varepsilon^8}\right\rceil
    \end{align}
    
    intervals $\{J_k\}_{k=1}^M$ such that for each $J_k$ we have
    \begin{align}
          \| Iu^\lambda\|_{L^8_t L^{10}_x(J_k\times\R)} \le \varepsilon.
    \end{align}
    
    By the almost Morawetz estimate and the bound
    \begin{align}
        \sup_{t\in[0, T'+\delta']} E(Iu^\lambda(t)) \le 1
    \end{align}
    \begin{align}
        \| Iu^\lambda\|_{L^{10}_{t,x}([0, T'+\delta']\times \R)}^{10} &\le C \|u_0^\lambda\|_{L^2_x}^3 + C \|u_0^\lambda\|_{L^2_x}^2 \sum_{k=1}^M \|I(u|u|^6) - Iu^\lambda |Iu^\lambda|^6\|_{L^1_t L^2_x(J_k\times \R)} \\
        &\le  C \lambda^{1/2}\|u_0\|_{L^2_x}^3 + C'\lambda^{1/3}\|u_0\|_{L^2_x}^2 \sum_{k=1}^M N^{-2+} \\
        &\le  C \lambda^{1/2}\|u_0\|_{L^2_x}^3 + C'\lambda^{1/3}\|u_0\|_{L^2_x}^2 \frac{(\lambda^2T_0)^{1/5} (2C \lambda^{1/2}\|u_0\|_{L^2_x}^3)^{4/5}}{\varepsilon^8} N^{-2+}
    \end{align}
    
For $s>29/79$, we may take
\begin{align} \label{N for Morawetz}
    N\ge C(s, \|u_0\|_{L^2_x})(1+T_0)^{\frac{6(s-1/6)}{79s-29-}}
\end{align}
and $C(s, \|u_0\|_{L^2_x})$ large enough such that
\begin{align}
    \| Iu^\lambda\|_{L^{10}_{t,x}([0, T'+\delta']\times \R)}^{10} \le \frac{3}{2} C \lambda^{1/2} \|u_0\|_{L^2_x}^3.
\end{align}
Therefore $Z$ is open and we have shown that
\begin{align}
    \| Iu^\lambda\|_{L^{10}_{t,x}([0, T+\delta]\times \R)}^{10} \le \frac{3}{2}C\lambda^{1/2} \|u_0\|_{L^2_x}^3.
\end{align}
\end{proof}
We now present the proof of the Main Theorem.
Just as in the proof above, we divide $[0, T+\delta]$ into
\begin{align}
    M=\left\lceil\frac{(\lambda^2T_0)^{1/5} ( 2C\lambda^{1/2}\|u_0\|_{L^2_x}^3)^{4/5}}{\varepsilon^8}\right\rceil
\end{align}
intervals $J_k=[t_k, t_{k+1}], k=0, \dots, M-1$ where $t_0=0$ and $t_M=T+\delta$ such that 
\begin{align}
    \| Iu^\lambda\|_{L^8_t L^{10}_x(J_k\times\R)} \le \varepsilon.
\end{align}
Then, we have that for any $t\in[0, T+\delta]$, there exists $k\in\{0, \dots, M-1\}$ such that $t\in J_k$. Hence,
\begin{align} \label{discrete FTC}
    E(Iu^\lambda(t)) &\le E(Iu^\lambda(0))+\sum_{j=0}^{k-1} |E(Iu^\lambda(t_{j+1})) - E(Iu^\lambda(t_{j}))| + |E(Iu^\lambda(t))- E(Iu^\lambda(t_k))|
\end{align}
By Theorem \ref{Energy Increment Theorem} and \eqref{discrete FTC}, we have that
\begin{align} 
     E(Iu^\lambda(t)) &\le \frac{1}{2} + \frac{(\lambda^2T_0)^{1/5} ( 2C\lambda^{1/2}\|u_0\|_{L^2_x}^3)^{4/5}}{\varepsilon^8} N^{-2+}\left( 1+\lambda^{1/6} \|u_0\|_{L^2_x} \right)
\end{align}
uniformly in $t\in[0,T+\delta]$.
If $s>39/89$, we may take 
\begin{align} \label{N for energy increment 2}
    N\ge C(s, \|u_0\|_{L^2_x}) (1+T_0)^{\frac{6(s-1/6)}{89s-39-\delta}}
\end{align}
for large enough $C(s, \|u_0\|_{L^2_x})$ so that
\begin{align}
    \frac{(\lambda^2T_0)^{1/5} ( 2C\lambda^{1/2}\|u_0\|_{L^2_x}^3)^{4/5}}{\varepsilon^8} N^{-2+}(1+ \lambda^{1/6} \|u_0\|_{L^2_x})\le \frac{1}{8}.
\end{align}
This implies that
\begin{align}
    \sup_{t\in[0, T+\delta]} E(Iu^\lambda(t))\le \frac{9}{10}
\end{align}
and thus the set $W$ is open and therefore
\begin{align}
    W=[0, \lambda^2 T_0].
\end{align}
Therefore,  for every $39/89<s<1$, it is enough to take 
\begin{align}
    N = C(s, \|u_0\|_{L^2_x}, \|u_0\|_{H^s_x}) (1+T_0)^{\frac{6(s-1/6)}{89s-39-\delta}}
\end{align}
to satisfy \eqref{N for Morawetz}, \eqref{N for energy increment 2} since for $s<1$ we have $\frac{6(s-1/6)}{89s-39}\ge \frac{6(s-1/6)}{79s-29}$. Now, using the scaling property
\begin{align}
    \| u(t)\|_{\dot{H}^s_x} = \lambda^{s-1/6} \| u^\lambda(\lambda^2t)\|_{\dot{H}^s_x},
\end{align}
the inequality
\begin{align}
    \| u^\lambda(t)\|_{H^s_x} \lesssim  \| Iu^\lambda(t)\|_{H^1_x},
\end{align}
and the following bound coming from the defocusing nature of the energy
\begin{align}
     \| Iu^\lambda(t)\|_{\dot{H}^1_x} \lesssim E(Iu^\lambda(t))^{1/2},
\end{align}
we have that for any $T\in[0, \infty)$
\begin{align}
     \sup_{t\in[0,T]}\| u(t)\|_{\dot{H}^s_x} &= \lambda^{s-1/6} \sup_{t\in[0,T]}\| u^\lambda(\lambda^2t)\|_{\dot{H}^s_x} \\
     &= \lambda^{s-1/6} \sup_{t\in[0,\lambda^2 T]}\| u^\lambda(t)\|_{\dot{H}^s_x} \\
    & \le \lambda^{s-1/6} \sup_{t\in[0,\lambda^2 T]}\| u^\lambda(t)\|_{H^s_x} \\
     &\lesssim  \lambda^{s-1/6} \sup_{t\in[0,\lambda^2 T]}\| Iu^\lambda(t)\|_{H^1_x} \\
     &\lesssim  \lambda^{s-1/6} \sup_{t\in[0,\lambda^2 T]}\| Iu^\lambda(t)\|_{L^2_x} +\lambda^{s-1/6} \sup_{t\in[0,\lambda^2 T]}\| Iu^\lambda(t)\|_{\dot{H}^1_x} \\
     &\lesssim \lambda^s \| u_0\|_{L^2_x} + \lambda^{s-1/6} \sup_{t\in[0,\lambda^2 T]} E(Iu^\lambda(t))^{1/2} \\
     &\lesssim \lambda^s \| u_0\|_{L^2_x} + \lambda^{s-1/6} \\
     &\lesssim (1+\| u_0\|_{L^2_x}) \lambda^s.
\end{align}
Therefore, with $\delta>0$ the loss in \eqref{N for energy increment 2}, 
\begin{align}
     \sup_{t\in[0,T]}\| u(t)\|_{H^s_x} \le C(s, \|u_0\|_{H^s_x}, \delta) (1+T)^{\frac{6s(1-s)}{89s-39-\delta}}, \quad \forall T\in[0, \infty)
\end{align}
by the conservation of mass $\|u(t)\|_{L^2_x}=\|u_0\|_{L^2_x}$. This concludes the proof of Theorem \ref{Main Theorem}.

\subsection*{AI Declarations}
Claude Fable 5.1 (Anthropic) was used to check the arguments and computations in Sections 3–6 and to proofread for grammar; the minor errors and omissions it identified were verified and corrected by the authors. ChatGPT 5.5 Pro (OpenAI) was used to help formulate and prove Proposition 4.2, a standard reduction that is implicit in earlier work and is included here for completeness; the statement and proof were verified and written by the authors. No AI assistance was used elsewhere.

\bibliographystyle{alpha}
\bibliography{SIAM-references_Lee_Yu}
\end{document}